\documentclass[12pt]{amsart}
\usepackage{geometry} 
\usepackage{pgfplots}
\usepackage{psfrag}
\usepackage{amsmath}
\usepackage{calc}
\usepackage{systeme}
\usepackage{amsfonts}
\usepackage{amsthm}
\usepackage{algorithm2e}
\usepackage[applemac]{inputenc}
\usepackage{bbold}
\usepackage[numbers]{natbib}

\numberwithin{equation}{section}

\newtheorem{Theorem}{Theorem}
\newtheorem{Lemma}{Lemma}[section]
\newtheorem{Rem}{Remark}[section]
\newtheorem{Def}{Definition}
\newtheorem{Prop}{Proposition}[section]
\newtheorem{Cor}{Corollary}[section]

\title{A spectral dominance approach to large random matrices}
\author{Charles Bertucci$^1$, M\'{e}rouane Debbah$^2$, Jean-Michel Lasry$^3$, Pierre-Louis Lions$^{3,4}$}
\thanks{$^1$ : CMAP, Ecole Polytechnique, UMR 7641, 91120 Palaiseau, France\\
$^2$ : Lagrange Mathematics and Computing Research Center, 75007, Paris\\
$^3$ :Universit\'e Paris-Dauphine, PSL Research University,UMR 7534, CEREMADE, 75016 Paris, France\\
$^4$ : Coll\`ege de France, 3 rue d'Ulm, 75005, Paris, France
}
\date{} 

\begin{document}

\maketitle
\begin{abstract}
This paper presents a novel approach to characterize the dynamics of the limit spectrum of large random matrices. This approach is based upon the notion we call "spectral dominance". In particular, we show that the limit spectral measure can be determined as the derivative of the unique viscosity solution of a partial integro-differential equation. This also allows to make general and "short" proofs for the convergence problem. We treat the cases of Dyson Brownian motions, Wishart processes and present a general class of models for which this characterization holds.
\end{abstract}

\tableofcontents
\section{Introduction}
This paper is the first of a series devoted to the systematic study of mathematical models describing the limiting dynamics of the spectrum of large random matrices. The present paper is concerned with the one dimensional case, i.e. the case of matrices having a real spectrum. We present a novel approach to characterize the limit spectrum as the spatial derivative of the unique viscosity solution of a certain partial differential equation (PDE in short). This approach is based upon a notion we call spectral dominance. Both the two dimensional case and extensions involving control or interactions between systems shall be presented in a future work.\\

The spectrum of several matrices, whose coefficients are given as solutions of stochastic differential equations (SDE in short), evolves as a system of interacting particles which tend to repel each other. Maybe the most famous example is given by the so-called Dyson Brownian motion, introduced in \citep{dyson} and which describes the evolution of the spectrum of a matrix whose coefficients are all independent real Brownian motions, except for the fact that the matrix is required to be symmetric.

Upon a proper rescaling, the evolution of the spectrum of such matrices becomes deterministic in the limit $N \to \infty$, where $N$ is the size of the matrix. Formally the limit spectral measure solves a non-local PDE of the form
\begin{equation*}{(*)}\label{eq0}
\partial_t µ + \partial_x(µH[µ]) = 0 \text{ in } (0,\infty)\times \mathbb{R},
\end{equation*}
for some non-local and singular operator $H$. For instance in the case of the Dyson Brownian motion, the operator $H$ is simply the Hilbert transform. As it is well known in the literature (see \citep{chan,rogers}), it is a quite general fact that the sequence of spectral measures is pre-compact for a certain topology and that all its accumulating points are weak solutions (in a sense which we do not define here) of ($*$). In the existing literature, uniqueness of solutions of ($*$) is established on a case by case basis using the specific nature of the operator $H$ \citep{chan,rogers,cabanal,li2020law}. In the present work, we propose a new (and more general) approach to characterize the limit spectrum $µ$ by considering its spatial primitive $F(t,x) := µ(t)((-\infty,x])$, which counts the number of eigenvalues below the level $x$ at time $t$. If $µ(t)$ is interpreted as a probability measure, then $F(t)$ is simply the associated cumulative distribution function. Formally if $µ$ solves ($*$), then $F$ solves

\begin{equation*}(**)
\partial_t F + (\partial_x F) \tilde{H}[F] = 0  \text{ in } (0,\infty)\times \mathbb{R},
\end{equation*}
where $\tilde{H}$ is defined by $\tilde{H}[F] := H[\partial_x F]$. It turns out that the PDE ($**$) satisfies a comparison principle result for a class of non-linearity $H$ that includes a wide range of models arising from the random matrix theory. This comparison principle allows us to use the theory of viscosity solutions which yields strong uniqueness and stability results on the solution and which thus permits us to characterize the unique solution of ($**$). Moreover, we are able to prove that the unique solution we characterize is the limit of the discrete model originating from the random matrix theory, when the size of the matrices tends to infinity.\\

At this point we emphasize the fact that the comparison property mentioned above has a discrete analogue for $N < \infty$ which is based upon the notion we call spectral dominance. We say that a symmetric $N\times N$ matrix $A$ is spectrally dominated by another symmetric $N\times N$ matrix $B$ if $\lambda_i(A) \leq \lambda_i(B)$ for all $1\leq i \leq N$, and we write $A \preceq B$, where we denote by $\lambda_1(C) \leq \lambda_2(C) ... \leq \lambda_N(C)$ the ordered eigenvalues of the symmetric matrix $C$. At the discrete level, the spectral dominance has already been used to prove the well-posedness of the interacting particles system \citep{anderson2010introduction,sniady2002random}.

Obviously, $A \preceq B$ if $A \leq B$ (namely $\lambda_i(A-B) \leq 0, \forall 1 \leq i \leq N$). Let us observe that $A \preceq B$ if and only if for all $x \in \mathbb{R}, \sharp\{i | \lambda_i(A) \leq x \} \leq \sharp\{i| \lambda_i(B) \leq x\}.$\\

In the first part of this paper we focus on the case of the Dyson Brownian motion to present in details the main ideas of this approach in this simplest setting. We then present how the results established can be generalized to the case of matrices of correlations, which can be refer to as the Wishart case in reference to the work \citep{wishart}. Finally we extend our results in an abstract framework in the last section of the paper. The rest of the introduction is devoted to bibliographical references.\\

Obviously, there exists a huge literature on the evolution of the spectrum of large random matrices and we are not going to present it exhaustively. As already mentioned above, the fact that the spectrum of matrices whose coefficients are driven by SDE can evolve as the empirical measure associated to a system of interacting particles is known since the work \citep{dyson}. Several other models have been studied, we can quote for instance the case of Wishart processes which is interested in correlation-like matrices \citep{bru1991wishart,allez2012invariant}. The convergence of the spectral measure in the limit of larger and larger matrices has been studied in \citep{chan,rogers}. The equation ($*$) has also been studied in several contexts in \citep{cabanal}. For (static) correlation matrices, a Wigner type law has been observed in \citep{marcenko}. More recently, a gradient flow like approach has been used to study equation ($*$), in the Dyson setting, in \citep{li2020law,donati2018convergence}. More generally we refer to \citep{anderson2010introduction} for a quite complete introduction to random matrices. Equation ($*$) has also been studied for several non-local terms $H$ in the context of Coulomb or Riesz gases in Physics, we refer to \citep{bolley2018dynamics} and references therein for more details on this models, which often take place in dimension $2$.\\

On the other hand, viscosity solutions have been introduced in \citep{crandall1983viscosity} for first order PDE. It has been generalized to second order fully non-linear PDE in \citep{jensen1988maximum,crandall1992user}. The case of integro-differential PDE has been the subject of more recent works such as \citep{awatif,barles2008second,arisawa2008remark}.\\

Let us end this bibliographical part by mentioning the link which exists between the theory of free probabilities and the theory of random matrices. It is well known that random matrices provide in some sense a canonical example of a space of free random variables on a certain free probability space. Thus, models such as the Dyson Brownian motion are closely related to free stochastic processes. For instance, the Fokker-Planck equation satisfied by the measure of a free Brownian motion is of the form of ($*$). We refer to \citep{biane1998stochastic} for more details on free stochastic calculus and to \citep{voiculescu1992free} for an introduction to the free probabilities theory. Although, for the most part, we shall not enter in this analogy, let us mention that the comparison principle that we mentioned earlier can be interpreted as a stochastic domination principle for free stochastic processes.

\section{The Dyson case}
\subsection{The model}
As mentioned in the introduction, we are only concerned in this section with the Dyson model that we now describe. Let $(\Omega, \mathcal{A}, \mathcal{F}, \mathbb{P})$ a standard filtered probability space. Let $N \geq 1$ be an integer and $x_0 <...< x_N$ be distinct real numbers. We call Dyson Brownian motion the family of stochastic processes $(\lambda^i)_{1\leq i \leq N}$ solutions of
\begin{equation}\label{dyson}
d\lambda^i_t = \frac{1}{N}\sum_{j \ne i} \frac{1}{\lambda^i_t - \lambda^j_t} dt + \frac{2}{\sqrt{N}}dB^i_t, \text{ for } 1\leq i \leq N,
\end{equation}
\begin{equation}\label{initial}
\lambda^i_0 = x_0,
\end{equation}
where $(B^i)_{1\leq N}$ is a family of independent Brownian motions on $(\Omega, \mathcal{A}, \mathcal{F}, \mathbb{P})$. Equation (\ref{dyson}) is satisfied (in law) by the eigenvalues of the matrix valued process $(A_t)_{t \geq 0}$ defined by 
\begin{equation}\label{processA}
d(A_{ij})_t = \frac{1}{\sqrt{N}} dW^{ij}_t \text{ for } 1 \leq i,j\leq N,
\end{equation}
where $(W^{ij})_{1\leq i,j\leq N}$ is a family of Brownian motions on $(\Omega, \mathcal{A}, \mathcal{F}, \mathbb{P})$ such that $W^{ij} = W^{ji}$ and $(W^{ij})_{1\leq i \leq j \leq N}$ is an independent family.

The mean field counterpart of (\ref{dyson}), namely the equation satisfied by the limit, as $N$ goes to $+ \infty$, of the empirical measure $\frac{1}{N}\sum_{i=1}^N \delta_{\lambda^i_t}$, is the following
\begin{equation}\label{ffp}
\partial_t µ + \partial_x(µ H[µ]) = 0 \text{ in } (0,\infty)\times\mathbb{R},
\end{equation}
where $H$, defined by
\begin{equation}\label{hilbert}
H[µ] (x) = \int_{\mathbb{R}}\frac{1}{x-y}µ(dy) \text{ on } \mathbb{R},
\end{equation}
is the Hilbert transform of $µ$ and the previous integral is understood in the sense of principal value. We expect that if $µ^N$, defined by 
\begin{equation}
µ^N(t) = \frac{1}{N}\sum_{i=1}^N\delta_{\lambda^i_t},
\end{equation}
converges toward some measure $µ$ when $N\to \infty$, then $µ$ is a solution of (\ref{ffp}). In this section we are mostly interested with the primitive equation of (\ref{ffp}) which is the following non-local transport equation :
\begin{equation}\label{teq}
\partial_t F + (\partial_x F) \tilde{H}[F] = 0 \text{ in } (0,\infty)\times \mathbb{R},
\end{equation}
where $\tilde{H}$ is the operator that corresponds to the squared root of $(-\frac{d^2}{dx^2})$, which is defined by 
\begin{equation}\label{halflapl}
\tilde{H}[\phi](x) = \int_{\mathbb{R}}\frac{\phi(x) - \phi(y)}{(x-y)^2} dy \text{ on } \mathbb{R}.
\end{equation}
Let us observe that the operator $\tilde{H}$ is well defined on functions $\phi$ which are bounded, with a bounded second order derivative (i.e. $\phi \in \mathcal{C}^{1,1}_b$). Indeed let us compute at some point $x\in \mathbb{R}$
\begin{equation}
\int_{|x -y| \leq 1}\frac{\phi(x) - \phi(y)}{(x-y)^2} dy = \int_{|x-y| \leq 1}\frac{\phi(x) - \phi(y) + \phi'(x)(x-y)}{(x-y)^2} dy,
\end{equation}
where we have used a symmetry argument. Since $\phi''$ is bounded by some constant $C$, we deduce from a Taylor-Lagrange formula that
\begin{equation}
\int_{|x -y| \leq 1}\frac{\phi(x) - \phi(y)}{(x-y)^2} dy \leq C.
\end{equation}
The rest of the integral defining $\tilde{H}[\phi](x)$ is also bounded since $\phi$ is bounded and $\mathbb{1}_{\{ |x| \geq 1\}} x^{-2}$ is integrable.
 Let us remark that for smooth function $\phi$ with compact support, the following holds
\begin{equation}
\tilde{H}[\phi] = H[\partial_x \phi],
\end{equation}
and both terms are well defined. We say that equation (\ref{teq}) is the primitive of (\ref{ffp}) because if $F$ is a smooth solution of (\ref{teq}) such that $\tilde{H}[F]$ is well defined, then $\partial_x F$ is a smooth solution of (\ref{ffp}).

An important feature of the operator $\tilde{H}$ is the following "maximum principle" property : if $\phi \in \mathcal{C}^{1,1}_b$ attains a maximum at $x_0$, then 
\begin{equation}\label{elliptic}
\tilde{H}[\phi](x_0) \geq 0.
\end{equation}

\subsection{The comparison principle and spectral dominance}
A quite interesting feature of the Dyson Brownian motion is that a certain comparison principle holds for solutions of (\ref{dyson}), and analogously at the continuous level for solutions of (\ref{teq}). To our knowledge, the first time such a comparison principle has been established was in a discrete setting \citep{sniady2002random} in a slightly different context (for singular values instead of eigenvalues). This idea has been used in \citep{anderson2010introduction} to prove the well-posedness of the system \eqref{dyson}.We provide the proof of the discrete comparison principle for the sake of completeness. Even though the generalization of a discrete comparison principle to the limit equation is natural and might be known from researchers in the field, we are not aware of such results.

\begin{Prop}\label{discretecomp}
Let $(B^i)_{1 \leq i \leq N}$ be a collection of Brownian motions and $(\lambda^i)_{1\leq i \leq N}$ and $(µ^i)_{1 \leq i \leq N}$ be two families of processes satisfying (\ref{dyson}) (in the strong sense). Assume that for all $i$, $\lambda^i_0 \leq µ^i_0$. Then for all time $t \geq 0$, for all $i$, $\lambda^i_t \leq µ^i_t$.
\end{Prop}
\begin{proof}
Let us recall that from \citep{chan,rogers}, almost surely, the trajectories $(\lambda^i)_{1\leq i \leq N}$ and $(µ^i)_{1 \leq i \leq N}$ are well defined, that is they are continuous and there is no collisions. Let us assume first that there exists $\epsilon > 0$ such that for all $i$,
\begin{equation}
\lambda^i_0 + \epsilon \leq µ^i_0.
\end{equation}
Let us define the random stopping time 
\begin{equation}
\tau_{\epsilon} := \inf \{ t, \exists i , \lambda^i_t  - µ^i_t> \epsilon\},
\end{equation}
for some deterministic $\epsilon > 0$. We now focus on the event $\{\tau_{\epsilon} < T\}$, for some deterministic and arbitrary $T > 0$. Let us define 
 \begin{equation}
 w^i_t = \lambda^i_t- µ^i_t - \delta t,
 \end{equation}
where $\delta := \epsilon / (2T)$. Remark that $w$ satisfies for all $t \geq 0$, $1\leq i\leq N$
\begin{equation}
\begin{aligned}
d w^i_t &= \left( \sum_{j\ne i} \frac{1}{\lambda^i_t - \lambda^j_t} - \frac{1}{µ^i_t - µ^j_t} \right)dt - \delta,\\
& = \sum_{j\ne i}\frac{µ^i_t- \lambda^i_t - (µ^j_t - \lambda^j_t)}{(\lambda^i_t - \lambda^j_t)(µ^i_t - µ^j_t)}dt - \delta.
\end{aligned}
\end{equation}
Hence it follows that the $w^i$ are smooths. Since $w^i_0 \leq 0$ for all $i$ and, conditioned on $\{\tau_{\epsilon} < T\}$, there exists $j$ such that $w^j_{\tau_{\epsilon}} > 0$, there exists $\tau$, random time such that 
\begin{equation}
\begin{cases}
\forall i,t \in [0,\tau] w^i_t \leq 0,\\
\exists i_0, w^{i_0}_{\tau} = 0,\\
dw^{i_0}_{\tau} \geq 0.
\end{cases}
\end{equation}
Computing $dw^i_t$ at $t = \tau$, $i = i_0$ (still conditioned on $\{\tau_{\epsilon} < T\}$), we deduce that $dw^{i_0}_{\tau} < -\delta$, which is a contradiction. Thus $\mathbb{P}(\{\tau_{\epsilon} < T\}) = 0$. We obtain that, almost surely, for all $i$ and $t\geq 0$
\begin{equation}
\lambda^i_t \leq µ^i_t + \epsilon.
\end{equation}
Since $\epsilon >0$ is arbitrary, the result easily follows.
\end{proof}
\begin{Rem}
The nature of the family of processes $(B^i)_{1 \leq i \leq N}$ is irrelevant since we are comparing two systems which are driven by the same $(B^i)_{1 \leq i \leq N}$.
\end{Rem}
The discrete comparison principle stated above can be reformulated in terms of the notations we previously introduced. Let $A$ and $B$ be two symmetric $d\times d$ matrices such that $A \preceq B$, and $(W_t)_{t \geq 0}$ a matrix valued process such that all its entries are independent Brownian motions, except for the fact that $W_t$ a symmetric matrix for all $t\geq 0$. Then for all $t\geq 0$, $A + W_t \preceq B + W_t$ in law, moreover, if $A$ and $B$ commutes, then the stochastic domination holds almost surely.\\

In the continuous setting, the comparison principle can be stated in the following form.
\begin{Prop}
Let $F_1$ and $F_2$ be two smooth bounded non decreasing solutions of (\ref{teq}) such that $F_1(0) \leq F_2(0)$, then for all time $t \geq 0$, $F_1(t) \leq F_2(t)$.
\end{Prop}
Since we are going to prove later more general results, we only sketch its proof.
\begin{proof}
We argue by contradiction and assume to simplify that there exists $(t_0,x_0)$ such that
\begin{equation}
\begin{cases}
\partial_t F_1(t_0,x_0) > \partial_t F_2(t_0,x_0),\\
F_1(t_0,x) \leq F_2(t_0,x), \forall x \in \mathbb{R},\\
F_1(t_0,x_0) = F_2(t_0,x_0).
\end{cases}
\end{equation}
Let us observe that the strict inequality in the first line can be obtained by adding $\delta t$ to $F_2$ in similar manner as what is done for the discrete comparison principle.
In particular, $\partial_x F_1(t_0,x_0) = \partial_xF_2(t_0,x_0)$ holds and, using the ellipticity of $\tilde{H}$, we obtain that
\begin{equation}
\tilde{H}[F_1(t_0)](x_0) \geq \tilde{H}[F_2(t_0)](x_0).
\end{equation}
Hence using the fact that $F_1$ and $F_2$ solves (\ref{teq}) at $(t_0,x_0)$, we arrive at a contradiction.

\end{proof}
\begin{Rem}
As we shall see later in the paper, this comparison principle holds in a much more general setting, without assuming the smoothness of $F_1$ and $F_2$.
\end{Rem}
\begin{Rem}
The monotonicity condition on the solution of (\ref{teq}) is both i) fundamental to have a weak sort of parabolic behaviour for (\ref{teq}) and ii) general for the problem we consider here.
\end{Rem}
In addition to establishing uniqueness of smooth increasing solutions of (\ref{teq}), the previous result is the main ingredient to apply the theory of viscosity solutions to study (\ref{teq}).\\

Let us end this section on the comparison principle and spectral dominance by a comment concerning the free probabilities theory. Although we do not introduce a precise free probability space nor present any proofs, we mention what we believe is the natural counter part of spectral dominance in free probabilities. A self adjoint free random variable has a law which is supported on the real line, hence a notion of cumulative distribution function can be introduced for which the notion of spectral dominance can naturally be adapted. Furthermore, the analogue of the comparison principle is that adding a free Brownian motion to free random variables preserves the spectral dominance. 

\subsection{Viscosity solutions of the primitive equation}
The theory of viscosity solutions, presented in \citep{crandall1983viscosity}, has already been extended to equations involving non-local operators, see for instance \citep{awatif,barles2008second,arisawa2008remark} for general aspects of this extension. The definitions and results we now give are merely adaptations of the works aforementioned to the case of (\ref{teq}). No a priori knowledge on viscosity solutions is needed to fully understand this section and we refer to \citep{crandall1992user} for more details on viscosity solutions.

The equation \eqref{teq} falls in the scope of the general theory only because we consider non decreasing solutions (in the spatial variable). At least formally, we can change \eqref{teq} in 
\begin{equation}\label{teq+}
\partial_t F + (\partial_x F)_+ \tilde{H}[F] = 0 \text{ in } (0,\infty)\times \mathbb{R}.
\end{equation}
This equation is a proper equation for the theory of viscosity solutions.

Let us insist briefly that, as we shall see in the last part of the paper, the following results are not particular to the Dyson case and shall be stated in more general framework later on. We begin with the following definition.

\begin{Def}\label{defvisc}
\begin{itemize}
\item  An upper semi continuous (usc in short) function $F$ is said to be a viscosity subsolution  of (\ref{teq+}) if for any smooth function $\phi \in \mathcal{C}^{1,1}_b$\footnote{$\mathcal{C}^{1,1}_b$ is the space of bounded functions with bounded second order derivatives.}, $(t_0,x_0) \in (0,\infty)\times \mathbb{R}$ point of strict maximum of $F - \phi$ the following holds
\begin{equation}
\partial_t \phi(t_0,x_0) + (\partial_x \phi(t_0,x_0))_+\tilde{H}[\phi(t_0)](x_0) \leq 0.
\end{equation}
\item A lower semi continuous (lsc in short) function $F$ is said to be a viscosity supersolution  of (\ref{teq+}) if for any smooth function $\phi \in \mathcal{C}^{1,1}_b$, $(t_0,x_0) \in (0,\infty)\times \mathbb{R}$ point of strict minimum of $F - \phi$ the following holds
\begin{equation}
\partial_t \phi(t_0,x_0) + (\partial_x \phi(t_0,x_0))_+\tilde{H}[\phi(t_0)](x_0) \geq 0.
\end{equation}
\item A viscosity solution $F$ of \eqref{teq+} is an usc viscosity subsolution such that $F_*$ is a supersolution where $F_*(t,x) = \underset{0 \leq s \to t,y\to x}{\liminf}F(s,y)$.
\item By extension, a function $F$ such that for all $t\geq 0$, $F(t)$ is non decreasing, is a viscosity solution of \eqref{teq} if it is a viscosity solution of \eqref{teq+}.
\end{itemize}
\end{Def}
\begin{Rem}\label{incretest}
When we are interested in non decreasing functions $F$ of the space variable, the definitions of sub and supersolutions are equivalent when one replaces the positive part of $\partial_x \phi(t_0,x_0)$ by $\partial_x \phi(t_0,x_0)$ itself. Indeed for any $t,h \geq 0$, if $x$ is a maximum of $F(t) - \phi(t)$, then
\begin{equation}
F(t,x) - \phi(t,x+ h) \leq F(t,x+h) - \phi(t,x+h) \leq F(t,x) - \phi(t,x),
\end{equation}
where we have used the fact that $F(t)$ is increasing in the first inequality. From this we easily deduce that $\lim_{h \to 0} (\phi(t,x+h) - \phi(t,x))h^{-1} \geq 0$ since $\phi$ is smooth. This justifies the last point of the previous definition.
\end{Rem}
An immediate result is that any smooth solution of \eqref{teq+} is also a viscosity solution of \eqref{teq+}. Indeed the following holds.
\begin{Prop}
Any bounded function $F$, such that $F,\partial_x F \in \mathcal{C}^{\alpha}$,\footnote{ $\mathcal{C}^{\alpha}$ is the space H\"older continuous function of exponent $\alpha$} for $\alpha > 0$, which is a solution of \eqref{teq+} is also a viscosity solution of \eqref{teq}.
\end{Prop}
\begin{proof}
The proof of this result is straightfroward so we only sketch it here. For such a function $F$, we immediately get that $\partial_t F \in \mathcal{C}^{\alpha}$. Thus $F$ is a classical solution of \eqref{teq+}, in the sense that all the terms in the equation make sense and that the equation is satisfied everywhere. It then suffices to use the ellipticity of $\tilde{H}$ to verify that $F$ is indeed a viscosity solution of \eqref{teq+}.
\end{proof}

One can easily adapt the general theory of viscosity solutions to obtain the two following results.
\begin{Theorem}\label{pll1}
For $F_0$ in $BUC(\mathbb{R})$ \footnote{ : $BUC(\Omega)$ is the space of bounded uniformly continuous functions on $\Omega$.} :
\begin{itemize}
\item There exists a unique $F \in BUC([0,T]\times \mathbb{R})$ for any $T > 0$, viscosity solution of \eqref{teq+} such that $F(0) = F_0$.
\item Provided that $F_0$ is lipschitz continuous, for any $t\geq 0$, $F$ satisfies $\|\partial_x F(t)\|_{\infty} \leq \|\partial_x F_0\|_{\infty}$.
\end{itemize}
\end{Theorem}
\begin{Theorem}
Let $F_1$ and $F_2$ be respectively bounded viscosity subsolution and supersolution of \eqref{teq+}. Then 
\begin{equation}
F_1|_{t = 0} \leq F_2|_{t = 0} \Rightarrow F_1 \leq F_2.
\end{equation}
\end{Theorem}
As we shall prove more general results in the last section of this paper, we postpone the proofs of those theorems.\\
As already mentioned above, the equation \eqref{teq} does not fall in the framework of viscosity solutions in general but it is the case when one is concerned with non decreasing solutions. In this context, as a consequence of the previous theorems and the remark \ref{incretest}, we can establish the following result concerning viscosity solutions of \eqref{teq}.
\begin{Cor}\label{corpll}
Let $F_1$ and $F_2$ be two bounded non-decreasing viscosity solutions of \eqref{teq}. Then
\begin{equation}
F_1|_{t = 0} \leq F_2|_{t = 0} \Rightarrow F_1 \leq F_2.
\end{equation}
\end{Cor}
The question of existence of viscosity solutions of (\ref{teq}) is treated in the next section on the convergence of the system of $N$ particles. Let us insist on the obvious fact that the previous result implies uniqueness of a viscosity solution of \eqref{teq} for a bounded non decreasing initial condition. 


We are not going to present in full details a theory of viscosity solutions of (\ref{teq}) or \eqref{teq+} as it would be merely a simple extensions of the existing literature aforementioned. However, let us state that the main advantage of the viscosity solution theory (besides the obtention of uniqueness results) is the strong stability properties that the comparison principle offers. For instance $L^{\infty}$ estimates are often derived from comparing a viscosity solution with either a constant or a more specific sub or supersolution.\\

\subsection{Convergence of the system of $N$ particles}
In this section we establish that the counting function $F_N$ of the $N$ eigenvalues system (\ref{dyson}), defined by
\begin{equation}
F_N(t,x) = N^{-1}\#\{i, \lambda^i_t \leq x\},
\end{equation}
converges toward a viscosity solution of (\ref{teq}) as $N \to \infty$. This result clearly justifies the use of the theory of viscosity solutions to treat equations such as (\ref{teq}).

\begin{Theorem}\label{convergence}
Assume that the empirical measure $µ_N^0$ of initial conditions of (\ref{dyson}) defined by
\begin{equation}
µ_N^0 = \frac{1}{N}\sum_{i = 1}^N \delta_{x_i}
\end{equation}
converges, almost surely, weakly toward a measure $µ_0$. Then, almost surely, the sequence $(F_N)_{N \geq 1}$ converges almost everywhere toward the unique viscosity solution $F$ of (\ref{teq}) which satisfies $F(0,x) = µ_0((-\infty,x])$ almost everywhere.
\end{Theorem}
\begin{proof}
From the results of \citep{chan,rogers}, we know that, almost surely, $(F_N)_{N \geq 1}$ is pre-compact in $\mathcal{C}([0,T],BV(\mathbb{R}))$ for any $T> 0$. We thus take a limit point $F$ of this sequence and to lighten the notations, we assume that the whole sequence converges toward the usc function $F$ (i.e. to consider the general case it suffices to replace all the $N$ in the rest of the proof by $\varphi(N)$ for some strictly increasing $\varphi: \mathbb{N}\to \mathbb{N}$).\\

We want to show that $F$ is a viscosity solution of (\ref{teq}). We only prove the subsolution property as the supersolution one can be proved following an analogous argument. (Recall that we have chosen $F$ to be usc). Let us take a function $\phi \in \mathcal{C}^{1,1}_b$ such that $F - \phi$ has a strict local maximum at $(t_0,x_0)$. We assume first that $\partial_x\phi(t_0,x_0) > 0$ and study the case $\partial_x \phi(t_0,x_0) =0$ later on. We assume that $\phi$ is a strictly increasing function with $\phi(-\infty) = 0$ and $\phi(\infty) = \beta \in[1, \infty)$.
 
We take $N$ large enough, $s_N>0$, chosen such that $s_N \to 0$ when $N \to \infty$, but as slowly as necessary. We define the system $(µ_i)_{1 \leq i}$ by 
\begin{equation}\label{defsystem}
\begin{cases}
µ^i_{t_0 -s_N} = \min(\lambda^i_{t_0-s_N}, (\phi(t_0-s_N))^{-1}(\frac{i}{N})),\\
dµ^i_t = \frac{1}{N}\sum_{j \ne i }\frac{1}{\mu^i_t - \mu^j_t} dt + \frac{2}{\sqrt{N}}dB^i_t,
\end{cases}
\end{equation}
 where the $(B^i)$ are the same brownian motions as in (\ref{dyson}).
 The system $(µ^i)$ is well defined. Let us note that there are $\lfloor\beta N\rfloor$ particles in $(µ^i)$ so that there might be more particles in $(µ^i)$ than in $(\lambda^i)$ but it does not matter as it does not alter the discrete comparison principle. By the discrete comparison principle (proposition \ref{discretecomp}), we know that for all $i, t\geq t_0 - s_N$, $µ^i_t \leq \lambda^i_t$. For all $N$, we consider an index $i_0$ (whose dependence in $N$ is not written  to simplify notations) which satisfies
 \begin{equation}
 \begin{cases}
 \limsup_{N \to \infty} \lambda^{i_0}_{t_0} \leq x_0,\\
 \frac{i_0}{N} \to F(t_0,x_0).
 \end{cases}
 \end{equation}
Let us remark that if choosing such an index is impossible, then $F_N$ does not converge toward $F$. Let us observe that
\begin{equation}
\phi(t_0-s_N,µ^{i_0}_{t_0 - s_N}) -\phi( t_0,µ^{i_0}_{t_0}) \geq \frac{i_0}{N} - \phi(t_0,\lambda^{i_0}_{t_0}),
\end{equation}
where we have used the definition of $µ^{i_0}$ and the fact that $\phi$ is increasing. Therefore, passing to the limit we deduce that 
 \begin{equation}
 \liminf_{N\to \infty}\phi(t_0-s_N,µ^{i_0}_{t_0 - s_N}) -\phi( t_0,µ^{i_0}_{t_0})\geq F(t_0,x_0) - \phi(t_0,x_0) = 0.
\end{equation}
 Hence, the following holds
 \begin{equation}\label{ineq}
 \liminf_{N \to \infty}\frac{\phi(t_0-s_N,µ^{i_0}_{t_0 - s_N}) -\phi( t_0,µ^{i_0}_{t_0})}{s_N} \geq 0.
 \end{equation}
 We now focus on the interactions between the $(µ^i)$. Recall (\ref{defsystem}) and let us assume first that for all $1\leq i \leq \lfloor \beta N\rfloor$
 \begin{equation}\label{easyhyp}
 µ^i_{t_0 -s_N} = (\phi(t_0-s_N))^{-1}(\frac{i}{N}).
 \end{equation}
 Thus, because the inverse function is decreasing on $\mathbb{R}_-$ and $\mathbb{R}_+$, we deduce that (at the time $t= t_0-s_N$) :
 
\begin{equation}\label{comphyp}
\frac{1}{N}\sum_{j \ne i_0} \frac{1}{\mu^{i_0} - \mu^j} \geq \frac{1}{N} \left( \frac{1}{µ^{i_0} - µ^{i_0 + 1}} + \frac{1}{µ^{i_0} - µ^{i_0 - 1}} \right) + \int_{(-\infty,µ^{\lfloor \beta N\rfloor}]\setminus [µ^{i_0 - 2}, µ^{i_0 + 1}]} \frac{\partial_x \phi(t_0-s_N, y)}{µ^{i_0} - y}dy.
\end{equation}
 Recalling that we have assumed that $\partial_x \phi(t_0,x_0) > 0$, a Taylor development of $\phi$ around $(t_0,x_0)$ yields that the first term of the right hand side vanishes as $N \to \infty$, while the regularity of $\phi$ around $(t_0,x_0)$ (and global in time continuity ) implies that the second term of the right hand side converges toward $\tilde{H}[\phi(t_0)](x_0)$ as $N \to \infty$. Thus we deduce : 
 \begin{equation}\label{conclinter}
 \liminf_{N \to \infty} \frac{1}{N}\sum_{j \ne i_0} \frac{1}{\mu^{i_0}_{t_0-s_N} - \mu^j_{t_0-s_N}} \geq \tilde{H}[\phi(t_0)](x_0).
 \end{equation}
 We now come back to the general case in which (\ref{easyhyp}) is not satisfied for all $i$. Let us recall that for any $\epsilon > 0$, $\phi(t- \epsilon) < F(t - \epsilon)$. Thus there exists $\gamma, \delta > 0$, such that $$\phi(t- \epsilon,x + y) \geq F(t- \epsilon, x + y) + \delta \text{ for all }|y| \leq \gamma.$$ Since $(F_N)_{N \geq 1}$ is a sequence of increasing functions converging toward $F$, the previous statement holds when replacing $F$ with $F_N$ and $\delta$ by $\delta/2$ when $N$ is large enough. Thus choosing $s_N$ as big as necessary (but still such that $s_N \to 0$ when $N \to \infty$), one obtain instead of (\ref{comphyp}) the following
 \begin{equation}
 \begin{aligned}
\frac{1}{N}\sum_{j \ne i_0} \frac{1}{\mu^{i_0} - \mu^j} \geq &\frac{1}{N} \left( \frac{1}{µ^{i_0} - µ^{i_0 + 1}} + \frac{1}{µ^{i_0} - µ^{i_0 - 1}} \right) + \int_{(-\infty,µ^{\lfloor \beta N\rfloor}]\setminus [µ^{i_0 - 2}, µ^{i_0 + 1}]} \frac{\partial_x \phi(t_0-s_N, y)}{µ^{i_0} - y}dy \\
& - \frac{1}{N}\left| \sum_{j, µ^j = \lambda^j}    \frac{1}{\mu^{i_0} - \mu^j} - \frac{1}{\mu^{i_0} - (\phi(t_0-s_N))^{-1}(\frac{j}{N})}  \right|.
\end{aligned}
\end{equation}
Recalling the convergence of $F_N$ toward $F$, we deduce that the last term in the previous expression vanishes as $N \to \infty$ and thus that (\ref{conclinter}) still holds.
 From \^{I}to's lemma, we obtain that
 \begin{equation}
 \begin{aligned}
 \phi(t_0,µ^{i_0}_{t_0}) - \phi(t_0 - s_n, µ^{i_0}_{t_0 - s_N}) = &\int_{t_0 - s_N}^{t_0}\partial_x\phi(t_0-t', µ^{i_0}_{t_0 - t'})\frac{1}{N}\sum_{j \ne i_0} \frac{1}{\mu^{i_0}_{t_0 -t'} - \mu^j_{t_0 - t'}}dt'\\
 & + \int_{t_0 - s_N}^{t_0}\partial_t\phi(t_0-t', µ^{i_0}_{t_0 - t'})dt'\\
  &+ \int_{t_0 - s_N}^{t_0}\frac{1}{2N^2} \partial_{xx} \phi(t_0 - t', µ^{i_0}_{t_0 -t'}) dt' \\
  &+  \frac{1}{N}\int_{t_0 - s_N}^{t_0}\partial_x\phi(t_0-t', µ^{i_0}_{t_0 - t'})dB^{i_0}_{t_0 - t'}
 \end{aligned}
 \end{equation}
Let us now remark that we can choose $s_N$ such that, almost surely, the last term in the previous expression is of order $o(s_N)$. Recalling (\ref{ineq}), dividing by $s_N$ and passing to the limit $N \to \infty$, we finally obtain that
 \begin{equation}\label{subvisc}
 \partial_t \phi(t_0,x_0) + \partial_x \phi(t_0,x_0)\tilde{H}[\phi(t_0)](x_0) \leq 0.
 \end{equation}
 Now let us come back to the case of a test function $\phi$ which is not a strictly increasing function with specific limits at $\pm \infty$. Because we assume that $\partial_x \phi(t_0,x_0) > 0$, there exists a strictly increasing function $\psi$ such that $\psi(-\infty) = 0, \psi(+ \infty) < \infty$, such that $\psi = \phi$ on a ball centered at $(t_0,x_0)$ and $F \leq \psi \leq \phi$ everywhere. Thus from the previous calculation, we know that $\psi$ satisfies 
  \begin{equation}
 \partial_t \psi(t_0,x_0) + \partial_x \psi(t_0,x_0)\tilde{H}[\psi(t_0)](x_0) \leq 0.
 \end{equation}
Moreover, the derivatives of $\psi$ and $\phi$ at $(t_0,x_0)$ are equal and the "ellepticity" of $\tilde{H}$ implies that $\tilde{H}[\phi(t_0)](x_0) \leq \tilde{H}[\psi(t_0)](x_0)$. Hence we deduce that (\ref{subvisc}) holds as soon as $\partial_x \phi(t_0,x_0) > 0$ and it remains to show it is true when $\partial_x \phi(t_0,x_0) = 0$.\\

 Let us now assume $\partial_x \phi(t_0,x_0) = 0$. Reasoning as in the previous part, it is enough to prove (\ref{subvisc}) holds for a function $\psi$ such that the first order derivatives of $\psi$ and $\phi$ coincides at $(t_0,x_0)$ and which satisfies $\psi \leq \phi$ with equality at $(t_0,x_0)$. Because $(t_0,x_0)$ is a strict global maximum of $F - \phi$, we know that there exists a function satisfying the previously mentioned requirements as well as $\phi \geq F$ everywhere and $\phi(t)$ is strictly increasing in $x$ for any $t < t_0$. Thus we can construct for all $N \geq0$ a system of particles $(µ^i)_{i \geq 1}$ associated to $\psi$ using (\ref{defsystem}). Let us now remark that since $\partial_x \psi(t_0,x_0) = 0$, 
 \begin{equation}
 \lim_{N \to \infty} \frac{1}{N(µ^{i_0}_{t_0-s_N} - µ^{i_0 + 1}_{t_0 - s_N})} = 0,
 \end{equation}
 where $i_0$ is an index depending on $N$ chosen as in the first part of the proof. Thus recalling that (at the time $t = t_0 - s_N$)
 
 \begin{equation}
\frac{1}{N}\sum_{j \ne i_0} \frac{1}{\mu^{i_0} - \mu^j} \geq \frac{1}{N} \left( \frac{1}{µ^{i_0} - µ^{i_0 + 1}}  \right) +\int_{(-\infty,µ^{\lfloor \beta N\rfloor}]\setminus [µ^{i_0 - 1}, µ^{i_0 + 1}]} \frac{\partial_x \psi(t_0-s_N, y)}{µ^{i_0} - y}dy,
\end{equation}
we obtain
  \begin{equation}
 \liminf_{N \to \infty} \frac{1}{N}\sum_{j \ne i_0} \frac{1}{\mu^{i_0}_{t_0-s_N} - \mu^j_{t_0-s_N}} \geq \tilde{H}[\psi(t_0)](x_0).
 \end{equation}
 Following the same argument as in the first case, we obtain that $F$ is indeed a viscosity solution of (\ref{teq}).\\
 
Let us now observe that the fact that $F(t)$ converges toward $µ_0((-\infty,x])$ when $t \to 0$ is rather easy to obtain (especially because the convergence is already known) following the arguments used in \citep{chan,rogers} to obtain the convergence results.\\

We now conclude by the uniqueness of such viscosity solutions that the whole sequence $(F_N)_{N \geq 1}$ converges toward $F$, the unique viscosity solution of \eqref{teq} which satisfies $F(0,x) = µ_0((-\infty,x])$ almost everywhere.

\end{proof}
\begin{Rem}
However long it may seem, the previous proof relies almost exclusively on the spectral dominance property (both discrete and continuous).
\end{Rem}

\subsection{General results on the Dyson case}
In this section we summarize the previous result in a compact manner. We also take advantage of this section to make links between the results we gave and the existing literature on this topic. In the previous sections, we have proven the

\begin{Theorem}
Given a probability measure $µ_0$ on $\mathbb{R}$, there exists a unique viscosity solution $F$ of \eqref{teq} (in the sense of definition \ref{defvisc}) which satisfies $F(0,x) = µ_0((-\infty,x])$.
\end{Theorem}
The following result is in fact a corollary of the proofs of the results above.
\begin{Cor}\label{potential}
Given a probability measure $µ_0$ on $\mathbb{R}$ and a continuous real valued function $B$ which satisfies
\begin{equation}
\exists c_0 > 0, \forall x,y \in \mathbb{R}, B(x) - B(y) \geq -c_0(x-y),
\end{equation}
there exists a unique viscosity solution $F$ of 
\begin{equation}
\partial_t F + B(x)\partial_x F+ (\partial_x F) \tilde{H}[F] = 0 \text{ in } (0,\infty)\times \mathbb{R},
\end{equation}
which satisfies $F(0,x) = µ_0((-\infty,x])$.
\end{Cor}
\begin{Rem}
We do not detail the way in which the term $B \partial_x F$ has to be understood in the viscosity sense as it is straightforward form the definitions we gave.
\end{Rem}
As usual in the viscosity solution theory, the addition of a linear first order term (or a potential, to use the terminology often used in the literature concerning such problems) is transparent in the study.
\begin{Rem}
Let us mention an extension of the result above. Assume that instead of being interested in the limit spectrum of a matrix valued process $(A_t)_{t \geq 0}$ as in \eqref{processA}, we are interested in the limit spectrum of $(\psi(A_t))_{t \geq 0}$ for a non-decreasing function $\psi : \mathbb{R} \to \mathbb{R}$ (where $(A_t)_{t \geq 0}$ is the same process as in \eqref{processA}). Then, since applying $\psi$ does not affect the order of the eigenvalues, one can simply study the limit spectrum of $(A_t)_{t \geq 0}$ and then consider its image by $\psi$ to get the limit spectrum of $(\psi(A_t))_{t \geq 0}$.
\end{Rem}

Let us now comment on how those results, compared to the existing literature. A problem which seems to have attracted quite a bit of focus in the last decades is the one of characterizing non smooth solutions of \eqref{ffp}. If the question of existence of weak solutions of \eqref{ffp} is quite clear since \citep{chan,rogers}, uniqueness of such weak solutions has attracted quite a lot of attention. The earlier results of uniqueness for weak solutions of \eqref{ffp} relied on either writing the equation satisfied by the Stieljes transform of the weak solution $µ$ \citep{rogers} or by considering the system of ordinary differential equations satisfied by the moments of the weak solutions \citep{chan}. In the first case, one obtain a complex Burgers equation and in the second case a triangular system which can be solved by induction. More recently, another approach has been developed in \citep{cabanal} and later on generalized in \citep{fontbona}. It consists in looking at the Fourier transform of the weak solution and realizing that this Fourier transform solves a particular PDE. (The paper \citep{cabanal} also uses the Stieljes transform and the characterization with moments to study the so-called Wishart and Unitary cases.) In all those cases, the addition of a potential as in Corollary \ref{potential} is either impossible or requires too strong assumptions on the potential. Let us mention that more recently, \eqref{ffp} has been studied with general potentials (as in Corollary \ref{potential}) in \citep{li2020law} following a gradient-flow like approach. This last approach relies on properties of the flow associated to \eqref{ffp} in Wasserstein metrics in the same fashion as the results we present in the next section.\\

To sum up, most of the results we presented in this section were already known in the literature (except for the fact that we are stating them on \eqref{teq} instead of \eqref{ffp}). However, all the proofs of the papers mentioned above seem to work only in the Dyson case (or the Wishart or Unitary cases) and do not extend to more general situations as in the next part of this paper. More than establishing new results on the Dyson case, we believe that we are providing a general framework which unifies the results known for either the Dyson or the Wishart case and allows to study even more general situations. (We are not going to study the Unitary case in this paper but we firmly believe that a similar approach can be followed in this context. We refer to \citep{cabanal} for more details on the Unitary case.)

\subsection{Digression : Contraction property of the Dyson flow in Wasserstein metric}
This section is devoted to a property of the flow generated by either the deterministic part of (\ref{dyson}) or (\ref{ffp}) that we believe to be useful in several contexts, even though we do not use it directly in this paper. The property we focus on is that the aforementioned flows are contractions in Wasserstein spaces. This could be quite easily generalized to other context but we restrict ourselves to the Dyson case in this paper. Moreover, as already mentioned, such results were already proved in \citep{li2020law}. Because our approach relies on quite different tools, we believe it to be worth mentioning.\\

Given $N \geq 1$, let us define the semigroup of operators $(T^N_t)_{t \geq 0}$ on $D^N:=\{x \in \mathbb{R}^N, x_1<...<x_N\}$ by
\begin{equation}
T^N_t x = (\lambda^i_t)_{1 \leq i \leq N},
\end{equation}
where $(\lambda^i)_{1 \leq i \leq N}$ are the solutions of 
\begin{equation}\label{deterflow}
\begin{cases}
\frac{d}{dt} \lambda^i_t = \sum_{j \ne i } \frac{1}{\lambda^i_t - \lambda^j_t} dt, \forall i;\\
\lambda^i_0 = x^i, \forall i.
\end{cases}
\end{equation}
Clearly these operators are well defined on $D^N$. Let us now remark that for any $p \in [1, \infty]$, for any $x, y \in \mathbb{R}^N$, 
\begin{equation}\label{lpwp}
\|x - y\|_p = \mathcal{W}_p\left(\frac{1}{N}\sum_{i = 1 }^N\delta_{x^i}, \frac{1}{N}\sum_{i = 1 }^N\delta_{y^i}\right),
\end{equation}
where $\mathcal{W}_p$ is the $p$ Wasserstein distance on measure supported on the real line. Thus knowing how $(T^N_t)_{t \geq 0}$ acts on $D^N$ with a $L^p$ distance is informative on how the Dyson flow acts on the space of empirical measure equipped with the associated Wasserstein distance. The following result shows that $(T^N_t)_{t \geq 0}$ is a family of contraction on $D^N$ equipped with $L^p$ distances.
\begin{Prop}
For any $N\geq 1, p \in [1,\infty], t \geq 0$, $T_t^N$ is a contraction on $(D^N,L^p)$.
\end{Prop} 
\begin{proof}
From proposition \ref{discretecomp}, we known that for any $N \geq 1, t \geq 0$, $T^N_t$ is order preserving, i.e. if $x,y$ are such that $x^i\leq y^i$ for all $i$, then 
\begin{equation}
(T^N_t x)^i \leq (T^N_t y)^i, \forall i.
\end{equation}
Let us define $\bold{1} := (1,..,1) \in \mathbb{R}^N$. Let us now remark that $T^N_t$ commutes with translation in $D^N$, i.e., for any $\alpha \in \mathbb{R}$ the following holds
\begin{equation}
T^N_t(x + \alpha \bold{1}) = T^N_t x + \alpha \bold{1}.
\end{equation}
 Thus using a result in \citep{crandall1980some}, we know that $T^N_t$ is a contraction in $(D^N, L^{\infty})$. Moreover, from (\ref{deterflow}), we deduce that for any $N \geq 1$, $x \in D^N$ :
\begin{equation}
\langle \bold{1},T^N_t x\rangle = \langle \bold{1}, x\rangle.
\end{equation}
Thus using once again a result in \citep{crandall1980some}, we obtain that $T^N_t$ is a contraction in $(D^N, L^1)$. Thus $T^N_t$ is almost everywhere differentiable and 
\begin{equation}
\max\left(\|DT^N_t(x)\|_1,\|DT^N_t(x)\|_{\infty}\right) \leq 1, \text{ a.e. in } D^N.
\end{equation}
Therefore by classical interpolation results, 
\begin{equation}
\|DT^N_t(x)\|_p \leq 1, \text{ a.e. in } D^N,
\end{equation}
which ends the proof.
\end{proof}
\begin{Rem}
The interested reader could easily check that the previous result only proceeds form the fact that the interaction between two particles is invariant if we translate the two particles ($L^{\infty}$ contraction), on the symmetry of the interaction ($L^1$ contraction), and obviously on the fact that a comparison principle can be stated (which has to do with decreasing properties of the interaction). Thus it holds true for instance for every interactions of the form $|x-y|^{\beta-1}(x-y)$ for $\beta < 0$.
\end{Rem}
\begin{Rem}
We could have formulated the previous result in the following way. Let us consider for $1 \leq i \leq N$, $\lambda^i_t$ as a function of the time and of the initial conditions $(\lambda^j_0)_{1 \leq j \leq N}$. The comparison principle implies that for all $t \geq 0$ and $1\leq i,j \leq N$, 
\begin{equation}\label{lambdaijrem}
\frac{d \lambda^i_t}{d \lambda^j_0} \geq 0.
\end{equation}
On the other hand, because the sum of the family $(\lambda^j_0)_{1 \leq j \leq N}$ is preserved through time, the sum over $j$ of the derivatives in \eqref{lambdaijrem} is equal to $1$. Thus, we de deduce that 
for all $t \geq 0$ and $1\leq i,j \leq N$, 
\begin{equation}
0 \leq \frac{d \lambda^i_t}{d \lambda^j_0} \leq 1.
\end{equation}

\end{Rem}
Using the previous result, we can now state its continuous analogue. For any probability measure $µ$ on $\mathbb{R}$, we define $\mathcal{T}_tµ$ as $\partial_x F(t)$ where $F$ is the unique viscosity solution of (\ref{teq}) which satisfies for almost every $x \in \mathbb{R}$, $F(0,x) = µ((-\infty,x])$. For any $t \geq 0$, $\mathcal{T}_t$ is clearly defined as an operator from $\mathcal{P}(\mathbb{R})$, the set of probability measures on $\mathbb{R}$, into itself. 
\begin{Prop}
The family $(\mathcal{T}_t)_{t \geq 0}$ is a family of contraction on $(\mathcal{P}(\mathbb{R}), W_p)$ for any $p\in [1,\infty]$.
\end{Prop}
\begin{proof}
Let $µ, \nu \in \mathcal{P}(\mathbb{R})$. For all $N \geq 1$, there exist $x_N, y_N \in D^N$ such that the associated empirical mean measures converges toward $µ$ and $\nu$. From the previous result, we know that for any $N\geq 1, t\geq 0, p \in [1,\infty]$
\begin{equation}
\|T^N_t x_N - T^N_t y_N\|_p \leq \|x_N - y_N\|_p.
\end{equation}
From theorem \ref{convergence} and \citep{chan,rogers}, we know that $(T^N_t x_N)_{N \geq 1}$ and $(T^N_t y_N)_{N \geq 1}$ weakly converge toward respectively $\mathcal{T}_t µ$ and $\mathcal{T}_t \nu$. Since the Wasserstein distances are metrics for the weak convergence of measures, we finally obtain that
\begin{equation}
\mathcal{W}_p(\mathcal{T}_t µ, \mathcal{T}_t \nu) \leq \mathcal{W}_p(µ, \nu).
\end{equation}
\end{proof}
\begin{Rem}
We believe it is worth mentioning that, at least formally, the previous result could have been established by looking at the PDE satisfied by the inverse cumulative distribution function $G$ defined for all $t$ by $G(t) = (F(t))^{-1}$. Indeed one can show that the $L^p$ distance between two inverse repartition functions associated to two measures is the $p$ Wasserstein distance between those two measures (it is the analogue of (\ref{lpwp}) for general measures). Moreover if $F$ satisfies (\ref{teq}) then $G$ solves formally
\begin{equation}
\partial_t G(t,x) = \int_{\mathbb{R}}\frac{1}{G(t,x) - G(t,y)}dy \text{ in } (0,\infty)\times \mathbb{R}.
\end{equation}
However studying directly this equation seems more difficult to justify, because of the singularity of $G$, than the passage to the limit we just presented.
\end{Rem}

\subsection{A comment on the technical difficulties arising from the free Fokker-Planck equation}
In this section, we insist on what we believe is the most challenging aspect of such systems and on why the theory of viscosity solutions is helpful in this context.

The main difficulty to study systems such as the Dyson Brownian motion is the singular nature of the interaction. Here the interaction between two particles becomes singular as the particles get closer and closer. In particular the interaction is not defined at range $0$ and this is clearly linked to the famous problem of the self interaction of an electron in Physics. When the measure describing the repartition of the particles is smooth (when it has a smooth density with respect to the Lebesgue measure for instance), the interactions between particles is well understood in the sense of principal value. However, this is not the case for general distribution of particles. In our opinion, this indeterminacy is best highlighted with the following example. Let us consider equation (\ref{ffp}) with initial condition $\delta_0$, the dirac mass at $0$. There are least two "natural" solutions $µ_1$ and $µ_2$ for this equation. The first one is given by $µ_1(t) = \delta_0$ for all time. It corresponds to the situation in which we model a sole particle, initially placed at $0$, which indeed creates a field given by $H[\delta_0]$ but is not affected by it since it does not interact with itself. The second solution $µ_2$ has a density which we still denote $µ_2$ and is given by
\begin{equation}
µ_2(t,x) = \frac{2}{\pi t}\sqrt{t - x^2}\mathbb{1}_{x\in [-\sqrt{t},\sqrt{t}]}, \text{ for } (t,x) \in (0,\infty)\times\mathbb{R}.
\end{equation}
This is the semi-circular law and it models the fact that particles which were initially concentrated around $0$ are going to spread because they are repelling each other. In the case of $µ_1$ it seems that no interaction occurs because there is no self interaction whereas it is clearly the case in the second example.

When looking at the microscopic model from which equation (\ref{ffp}) arises (in the application we are interested in), it is clear that the second option is more natural. Indeed for a finite number of particles, each particles are subject to independent brownian motions which also has the effect to spread the particles instantly, and thus prevent situations such as $µ_1$ to happen in the limit of a large number of particles. In some sense, a contribution of this paper is to prove that the theory of viscosity solutions allows us to choose the correct solution of equation (\ref{ffp}). We leave as an exercise to the interested reader that the spatial primitive of $µ_2$ is indeed a viscosity solution of (\ref{teq}) whereas it is not the case for $µ_1$.

\section{The Wishart case}
We now present the extensions of the previous part to the case of the limit of Wishart processes. 
\subsection{The model}
The model we are interested in, in the present section, can be introduced at the discrete level as the following. Let $(A_t)_{t \geq 0}$ be a $n\times m$ matrix valued process such that all its coefficients are independent real Brownian motions on a filtered probability space $(\Omega, \mathcal{A},\mathcal{F},\mathbb{P})$. We consider the $n\times n $ matrix valued process $(X_t)_{t \geq 0}$ defined by 
\begin{equation}\label{aat}
X_t = \frac{1}{n}A_t A_t^T,
\end{equation}
where $A^T$ is the transpose of the matrix $A$. We shall assume in the rest of this section that $m \geq n$. Just as in the Dyson case, the eigenvalues $(\lambda^i_t)_{1\leq i \leq n}$ of the process $(X_t)_{t\geq 0}$ satisfy a system of SDE. To our knowledge this system has first been derived in \citep{bru1991wishart} and it is the following.
\begin{equation}\label{mp}
d\lambda^i_t = \left( \frac{m}{n} + \frac{1}{n} \sum_{j \ne i}\frac{\lambda^i_t + \lambda^j_t}{\lambda^i_t - \lambda^j_t}\right)dt + \frac{2}{n}\sqrt{\lambda^i_t}dB^i_t, \forall 1\leq i \leq n,
\end{equation}
where $(B^i)_{1\leq i \leq n}$ is a collection of independent Brownian motions on $(\Omega, \mathcal{A}, \mathcal{F}, \mathbb{P})$. This system of SDE is the analogue of (\ref{dyson}). Clearly the $(\lambda^i)_{1\leq i \leq n}$ are expected to be positive and let us observe that this is the case for the solutions of the previous system in the case $m\geq n$. In the mean field limit (i.e. when $n \to \infty$) we can also derive a mean field equation for the limit spectral measure, but the behavior of $m$ as $n \to \infty$ has to be prescribed. We assume in the rest of this section that there exists $c\geq 1$ such that
\begin{equation}
\frac{m}{n} \underset{n \to \infty}{\longrightarrow} c.
\end{equation}
\begin{Rem}
The case $c < 1$ can be treated by looking at $A_t^T A_t$ in \eqref{aat} instead of $A_t A_t^T$. More general assumptions (such as a non constant $c$) shall be the subject of future works by the authors.
\end{Rem}
The mean-field analogue of (\ref{mp}) is the following PDE
\begin{equation}\label{mpffp}
\partial_t µ + \partial_x\left(µ \mathcal{K}[µ(t)](x)\right) = 0 \text{ in } (0,\infty)^2,
\end{equation}
where the operator $\mathcal{K}$ is defined for smooth integrable functions $\phi$ by
\begin{equation}
\mathcal{K}[\phi](x) = c +  \int_{\mathbb{R}_+}\frac{x + y}{x -y}\phi(y)dy.
\end{equation}
This last integral is understood in the sense of principal value. There are two notable differences with the Dyson case here. The first one is that there is a boundary condition at $x=0$, the eigenvalues being forced to be positive. The second one is that the non local part in $\mathcal{K}$ is not a convolution, but only a kernel operator. In terms of interactions between particles, this means that the interactions between particles depends on where the particles are. Following the same ideas as the ones of the Dyson case, we are interested in the primitive equation of (\ref{mpffp}). This equation is given by
\begin{equation}\label{mpteq}
\partial_t F + (\partial_x F) \tilde{\mathcal{K}}[F] = 0 \text{ in } (0,\infty)^2,
\end{equation}
where $\tilde{\mathcal{K}}$ is the operator defined on smooth functions with finite limit at $+ \infty$ by
\begin{equation}
\tilde{\mathcal{K}}[\phi](x) = c - \phi(+\infty) - \phi(0) + 2x\int_{\mathbb{R}_+}\frac{\phi(x) - \phi(y)}{(x-y)^2} dy.
\end{equation}
The integral is understood in the sense of principal values (when $x>0$). Several comments can be made on this operator. The first one is that the term $- \phi(+\infty) - \phi(0)$, however uncommon, does not raise any difficulty a priori as we mainly intend to evaluate $\tilde{\mathcal{K}}$ on functions $F$, which are all supposed to satisfy $F(0) = 0$ and $F(+ \infty) = 1$. Thus, the constant terms in $\tilde{\mathcal{K}}$ should be equal to $c-1 > 0$, which justifies the fact that a priori, no boundary conditions is needed at $x=0$. Moreover, the last term of $\tilde{\mathcal{K}}$ is simply the same term as $\tilde{H}$ that we studied in the previous part except for the facts that it is multiplied by $2x$ and that the integral is taken only over $\mathbb{R}_+$. This term is well defined, even for $x = 0$ as we shall see in the next section. Finally, $\tilde{\mathcal{K}}$ satisfies the ellipticity condition : for smooth functions $\phi$ and $\psi$ in the domain of $\tilde{\mathcal{K}}$ such that $\phi \leq \psi$ and $\phi(x_0) = \psi(x_0)$ for some $x_0 \in \mathbb{R}_+$, then
\begin{equation}
\tilde{\mathcal{K}}[\phi](x_0) \leq \tilde{\mathcal{K}}[\psi](x_0).
\end{equation}

\subsection{The boundary at $x=0$}

From a mathematical point of view, an important feature of the Wishart case, compared to the Dyson case, is the presence of a boundary condition at $x=0$. We shall not provide detailed proofs concerning the Wishart case and simply refer to the next part, where proofs are given in a more general framework, except for the fact that the next section is set in the whole $\mathbb{R}$. Since we adopt this strategy of presentation, we explain in this section why this boundary does not raise any particular difficulties.\\

First let us state that the integral term in $\tilde{\mathcal{K}}$ is well defined for $x=0$. Indeed, given a smooth function $\phi$ in the domain of $\tilde{\mathcal{K}}$, for any $x> 0$, the integral is well defined in the sense of principal values and
\begin{equation}\label{k0}
\lim_{x \to 0^+} x\int_{\mathbb{R}_+}\frac{\phi(x) - \phi(y)}{(x-y)^2} dy = 0.
\end{equation}

Secondly, let us mention that because we are in the case $c\geq1$, there is no accumulation of mass near $x= 0$. More precisely, it can easily be shown that solutions $F$ of (\ref{mpteq}) are bounded from above by a continuous function $S$ such that $S(t,0) = 0$ for all $t\geq 0$.\\

Finally let us mention that, should we have place ourselves in the case $c<1$, the situation would have been entirely different. Indeed in this latter case, an accumulation of eigenvalues at $0$ is expected. This can be easily understood by looking at the rank of the matrix $X_t$ in this case. To illustrate this phenomena, let us recall the Marcenko-Pastur distributions, which are stationary states of (\ref{mpffp}), when one adds a suitable confinement potential. They are parametrized by $\sigma$ and depend on $c$. They are given by
\begin{equation}
µ_{\sigma,c}(x) = \frac{c\sqrt{(\lambda_+ -x)(x-\lambda_-)}}{2\pi\sigma^2 x}\mathbb{1}_{[\lambda_-,\lambda_+]}(x),
\end{equation} 
in the case $c \geq  1$, and by
\begin{equation}
µ_{\sigma,c} = (1-c)\delta_0 + \frac{c\sqrt{(\lambda_+ -x)(x-\lambda_-)}}{2\pi\sigma^2 x}\mathbb{1}_{[0,\lambda_+]}(x)dx,
\end{equation}
in the case $c<1$, where $\lambda_{\pm} = \sigma^2(1 \pm \sqrt{c^{-1}})^2$.

\subsection{Comparison principle and viscosity solutions}
In this section we present a comparison principle for viscosity solutions of
\begin{equation}\label{mpteq}
\partial_t F + (\partial_x F)_+ \tilde{\mathcal{K}}[F] = 0 \text{ in } (0,\infty)^2.
\end{equation}
As it was already the case in the previous section, the positive part in the previous equation can be removed when one is concerned with non decreasing solutions of the space variable (which is the case when one is studying the spectrum of large random matrices). We recall the definitions of viscosity sub and supersolutions :
\begin{Def}
\begin{itemize}
\item  An usc function $F$ is said to be a viscosity subsolution  of (\ref{mpteq}) if for any smooth function $\phi \in \mathcal{C}^{1,1}_b$ with limits at $0$ and $+ \infty$, $(t_0,x_0) \in (0,\infty)^2$ point of strict maximum of $F - \phi$ the following holds
\begin{equation}
\partial_t \phi(t_0,x_0) + (\partial_x \phi(t_0,x_0))_+\tilde{\mathcal{K}}[\phi(t_0)](x_0) \leq 0.
\end{equation}
\item A lsc function $F$ is said to be a viscosity supersolution  of (\ref{mpteq}) if for any smooth function $\phi \in \mathcal{C}^{1,1}_b$ with limits at $0$ and $+ \infty$, $(t_0,x_0) \in (0,\infty)^2$ point of strict minimum of $F - \phi$ the following holds
\begin{equation}
\partial_t \phi(t_0,x_0) + (\partial_x \phi(t_0,x_0))_+\tilde{\mathcal{K}}[\phi(t_0)](x_0) \geq 0.
\end{equation}
\item A viscosity solution $F$ of \eqref{mpteq} is an usc viscosity subsolution such that $F_*$ is a supersolution where $F_*(t,x) = \underset{0 \leq s \to t,y\to x}{\liminf}F(s,y)$.
\end{itemize}
\end{Def}

\begin{Rem}
Let us remark that because (\ref{k0}) holds, we could have allowed a slightly more sophisticated definition of viscosity solutions by allowing the point of maximum to be such that $x_0 = 0$.
\end{Rem}
The following holds.
\begin{Prop}
Let $F_1$and $F_2$ be respectively viscosity subsolution and supersolution of (\ref{mpteq}). Let us assume that for any $t>0$, $(F_1(s))_{0\leq s\leq t}$ and $(F_2(s))_{0\leq s\leq t}$ are such that they converge uniformly in $s$ toward $0$ when $x\to 0$ and uniformly in $s$ toward $1$ when $x\to \infty$. Let us also assume that $F_1(0) \leq F_2(0)$, then for all $t\geq 0$, $F_1(t) \leq F_2(t)$.
\end{Prop}
\begin{proof}
The proof of this statement is simply a mild adaptation of the proof of proposition \ref{compg}, which is given in full details. Hence we do not give this proof in its full length and only comment on the slight changes between the two statements. The main difference lies in the presence of the boundary condition and on the fact that the kernel is here unbounded (it is here simply given by $g(x,z) = x$, referring to the notations of the next section). The behavior we prescribe for $F_1$ and $F_2$ clearly prevent any difficulties which may come from this changes.
\end{proof}
Let us comment on the behavior we impose on $F_1$ and $F_2$ in the previous statement. Those assumptions are natural when having in mind that we are interested in studying the limit behavior of (\ref{mp}) when $N\to \infty$. Those two assumptions are then simply the fact that $F_1$ and $F_2$ represents counting functions of systems for which there is no aggregation at $0$ or loss of mass at infinity. Obviously more general statement could have been made as the next remark shows.
\begin{Rem}
The assumptions on the limits of $F_1$ and $F_2$ could have been replaced by $(F_1(s)-F_2(s))_{0\leq s\leq t}$ converges uniformly in $s$ toward negative limits at $x\to 0$ and $x \to \infty$.
\end{Rem}
As a consequence of the previous comparison principle, we can state the following result of uniqueness.
\begin{Theorem}
Given a non decreasing function $F_0$ such that $F_0(0_+) = 0$ and $F_0(+\infty) = 1$, there exists at most one viscosity solution of (\ref{mpteq}) such that $F(0) = F_0$.
\end{Theorem}
\begin{proof}
The proof of this result is a mild adaptation of the argument of the proof of theorem \ref{uniqg}, that we do not present here.
\end{proof}
\begin{Rem}
Just as it was true for the comparison principle, the previous result could easily be extended to more general boundary conditions than $0$ and $1$ for the respective limits at $0$ and $ + \infty$. Moreover, time dependent boundary conditions for (\ref{mpteq}) shall be the subject of a future work by the authors of this paper.
\end{Rem}
\subsection{Existence of viscosity solutions of the transport equation}
Although we mainly leave open the question of existence of viscosity solutions of (\ref{mpteq}), let us comment on this question. The argument we presented in the proof of theorem \ref{convergence} depends quite weakly on the nature of the interactions between the particles, except for the fact that it preserves a comparison principle. Thus this argument can easily be adapted to the present case. The main difficulty one would have to deal with in order to establish an existence result following the same idea would be to prove compactness results such as the ones of \citep{chan,rogers} which are not to our knowledge already known in the literature. Such results could be quite easily obtained by changing mildly the proofs of the Dyson case, but with some technical difficulties which we do not believe are helpful for the present paper, even though they may present an interest in themselves.

\section{A general framework}
In this last part we present a set of operators $\mathcal{L}$ for which we are able to prove uniqueness of viscosity solutions of the following non local transport equation
\begin{equation}\label{teqg}
\partial_t F + (\partial_x F) \mathcal{L}[F] = 0 \text{ in } (0,\infty)\times \mathbb{R},
\end{equation}
\begin{equation}
F(0) = F_0 \text{ in } \mathbb{R}.
\end{equation}
We shall focus in this section on operators of the form 
\begin{equation}
\mathcal{L}[\phi](x) = \int_{\mathbb{R}}\frac{g(x,z)(\phi(x) - \phi(x + z))}{z^2}dz,
\end{equation}
defined for smooth functions $\phi$ where $g$ is a function on which assumptions shall be made later on. Formally equation (\ref{teqg}) is linked with the mean field transport equation associated to a system of particles which interact as
\begin{equation}\label{generalsystem}
d\lambda^i_t = \frac{1}{N}\sum_{j\ne i}\frac{f(\lambda^i_t,\lambda^j_t)}{\lambda^i_t - \lambda^j_t} dt + \epsilon_N dB^i_t,
\end{equation}
where $\epsilon_N = o(1)$, $(B_t)_{t \geq 0}$ is a $N$ dimensional Brownian motion and $f$ and $g$ satisfy
\begin{equation}\label{gf}
\forall x,y \in \mathbb{R}, (x-y)\partial_y f(x,y) + f(x,y) = g(x,y-x).
\end{equation}
The Dyson case of course corresponds to $f \equiv 1$ and the Wishart case to $f(x,y) = x + y$. In general, several functions $f$ can be of interest. Let us derive such $f$ on an example that we believe to be quite instructive. Consider a $N\times N$ matrix valued diffusion process $(A_t)_{t \geq 0}$ which satisfies for $t\geq 0$
\begin{equation}\label{generaldiff}
dA_t = \sigma(A_t)dW_t\sigma(A_t),
\end{equation}
where $(W_t)_{t\geq 0}$ is a $N\times N$ Dyson Brownian motion, $\sigma : \mathbb{R} \to \mathbb{R}$ is a smooth function and $A_0$ is assumed to be a symmetric matrix. A perturbation theory computation yields that, if we denote by $(\lambda^i_t)_{1 \leq i \leq N}$ the ordered spectrum of $A_t$,
\begin{equation}
d\lambda^i_t = \sum_{j\ne i}\frac{\sigma(\lambda^i_t)\sigma(\lambda^j_t)}{\lambda^i_t - \lambda^j_t} dt + dB^i_t,
\end{equation}
where $(B_t)_{t \geq 0}$ is a $N$ dimensional Brownian motion. The following extension is also worth mentioning. By considering several functions $\sigma_k$, Dyson Brownian motions $W^k$ and the diffusion given by
\begin{equation}
dA_t = \sum_{k} \sigma_k(A_t)dW^k_t\sigma_k(A_t),
\end{equation} 
we obtain that the associated system of eigenvalues satisfies
\begin{equation}
d\lambda^i_t = \sum_{j\ne i}\frac{\sum_k \sigma_k(\lambda^i_t)\sigma_k(\lambda^j_t)}{\lambda^i_t - \lambda^j_t} dt + d\tilde{B}^i_t,
\end{equation}
where $(\tilde{B}_t)_{t \geq 0}$ is, up to a multiplicative constant, a $N$ dimensional Brownian motion. Hence general types of interactions naturally arise when one consider diffusion such as (\ref{generaldiff}), and this leads to a general class of non-local operator $\mathcal{L}$. Let us mention that in the context of free probabilities, (\ref{generaldiff}) can be interpreted as a diffusion with a non constant volatility and that in this situation, we naturally expect the cumulative distribution function of the process to satisfy (\ref{teqg}) with the appropriate operator $\mathcal{L}$.\\

Concerning the link between (\ref{generalsystem}) and (\ref{teqg}), we do not enter into much details and only mention that formally (\ref{teqg}) is the PDE satisfied by the spatial primitive of the limit when $N \to \infty$ of the sequence of empirical measure of solutions of (\ref{generalsystem}). The relation (\ref{gf}) implies in particular that $g(x,z)$ should satisfy
\begin{equation}\label{derg}
\partial_zg(x,0) = 0, \forall x \in \mathbb{R}.
\end{equation}
The natural condition under which the comparison principle should hold (for either (\ref{generalsystem}) or (\ref{teqg})) is that 
\begin{equation}\label{hypg}
\forall x,z \in \mathbb{R}, g(x,z) \geq 0.
\end{equation}
In the Dyson case, we simply had $g(x,z) = 1$ uniformly and $g(x,z) = 2x\mathbb{1}_{\{z \geq -x\}}$ in the Wishart  case (omitting the question of the boundary condition).
As we shall see, under an additional smoothness assumption on $g$, the comparison principle indeed holds under (\ref{hypg}). More surprisingly, we shall also see that well posedness of (\ref{teqg}) is merely a consequence of the positivity of $g$ along the $x$ axis (i.e. $g(x,0) \geq 0$), even though, in this more general situation, the comparison principle does not hold.\\

\subsection{Viscosity solutions of the primitive equation}
As we did in the previous cases, we introduce the parabolic equation 
\begin{equation}\label{teqg+}
\partial_t F + (\partial_x F)_+ \mathcal{L}[F] = 0 \text{ in } (0,\infty)\times \mathbb{R}.
\end{equation}
We start with the usual definitions of viscosity solutions.

\begin{Def}
\begin{itemize}
\item  An usc function $F$ is said to be a viscosity subsolution  of (\ref{teqg+}) if for any smooth function $\phi \in \mathcal{C}^{1,1}_b$, $(t_0,x_0) \in (0,\infty)\times \mathbb{R}$ point of strict maximum of $F - \phi$ the following holds
\begin{equation}
\partial_t \phi(t_0,x_0) + (\partial_x \phi(t_0,x_0))_+\mathcal{L}[\phi(t_0)](x_0) \leq 0.
\end{equation}
\item A lsc function $F$ is said to be a viscosity supersolution  of (\ref{teqg+}) if for any smooth function $\phi \in \mathcal{C}^{1,1}_b$, $(t_0,x_0) \in (0,\infty)\times \mathbb{R}$ point of strict minimum of $F - \phi$ the following holds
\begin{equation}
\partial_t \phi(t_0,x_0) + (\partial_x \phi(t_0,x_0))_+\mathcal{L}[\phi(t_0)](x_0) \geq 0.
\end{equation}
\item A viscosity solution $F$ of \eqref{teqg+} is an usc viscosity subsolution such that $F_*$ is a supersolution where $F_*(t,x) = \underset{0 \leq s \to t,y\to x}{\liminf}F(s,y)$.
\item By extension, a function $F$ such that for all $t\geq 0$, $F(t)$ is non decreasing, is a viscosity solution of \eqref{teqg} if it is a viscosity solution of \eqref{teqg+}.
\end{itemize}
\end{Def}

We also introduce an equivalent reformulation of the notion of sub and super solutions of (\ref{teqg+}). Such formulations are frequent in the literature on viscosity solutions, see \citep{awatif,barles2008second}. We introduce to operators $\mathcal{I}_{1,\delta}$ and $\mathcal{I}_{2,\delta}$ with the following
\begin{equation}
\mathcal{I}_{1,\delta}[\phi](x) = \int_{|z|\leq \delta}\frac{g(x,z)(\phi(x) - \phi(x + z))}{z^2}dz
\end{equation}
\begin{equation}
\mathcal{I}_{2,\delta}[\phi](x) = \int_{|z| > \delta}\frac{g(x,z)(\phi(x) - \phi(x + z))}{z^2}dz
\end{equation}

\begin{Prop}
Let $F$ be a subsolution (resp. supsolution) of (\ref{teqg+}). Then for all smooth functions $\phi$, $\delta > 0$ and $(t_0,x_0)\in (0,\infty)\times \mathbb{R}$ such that i) $(F - \phi)(t_0,x_0) = 0$, ii) $(F - \phi)(t,x) \leq 0$ (resp. $\geq 0$) for any $(t,x)\in B((t_0,x_0),\delta)$, the following holds
\begin{equation}
\begin{aligned}
\partial_t \phi(t_0,x_0) + (\partial_x \phi(t_0,x_0))_+( \mathcal{I}_{1,\delta}[\phi(t_0)](x_0)+\mathcal{I}_{2,\delta}[F(t_0)](x_0) ) \leq 0 \text{ (resp. } \geq 0 \text{)}.
 \end{aligned}
\end{equation}
\end{Prop}
The proof of the previous statement is an immediate adaptation of a result in \citep{awatif} that we do not detail here.
\subsection{A comparison principle for Lipschitz solutions}
To establish a comparison principle, we assume the following on $g$.
\begin{itemize}
\item $g(x,z)\geq 0$ on $\mathbb{R}^2$.
\item There exists $C,\alpha_0> 0$ such that 
\begin{equation}\label{req1}
|g(x,z) - g(y,z)| \leq C |x-y|, \text{ for } x,y,z \in \mathbb{R}.
\end{equation}
\begin{equation}\label{req2}
C^{-1} \leq g \leq C \text{ on } \mathbb{R}\times(-\alpha_0,\alpha_0).
\end{equation}
\begin{equation}\label{req3}
\left| \frac{\partial_x g(x,z)}{g(x,z)} - \frac{\partial_xg(x,0)}{g(x,0)}\right| \leq C|z| \text{ on } \mathbb{R}\times(-\alpha_0,\alpha_0).
\end{equation}
\end{itemize}
In particular let us remark that we do not assume that (\ref{derg}) holds. We now prove a comparison principle when one of the two functions is Lipschitz in space. We shall provide a general comparison principle later on.
\begin{Lemma}\label{complip}
Assume that $F_1$ and $F_2$ are respectively viscosity subsolution and supersolution of (\ref{teqg+}) such that one of them is Lipschitz continuous in space, locally uniformly in time. If $F_1(0) \leq F_2(0)$, then for all time $F_1(t) \leq F_2(t)$.
\end{Lemma}
\begin{proof}
This, rather technical, proof is classical in the theory of viscosity solutions, see \citep{awatif,barles2008second}. It does not rely on particular new ideas and we mostly present it in details for the sake of completness. We assume without loss of generality that $F_2$ is Lipschitz continuous in space. We argue by contradiction and use the usual technique of doubling of variables. Let us assume that the result is false. Thus there exists $\alpha > 0$ and $T>0$ such that for any $\epsilon > 0$,
\begin{equation}\label{sup}
\sup \left\{F_1(t,x) - F_2(s,y) - \frac{1}{2\epsilon}(x-y)^2 - \frac{1}{2\epsilon}(t-s)^2 | t,s \in [0,T] , x,y \in \mathbb{R} \right\} > \alpha,
\end{equation}
Since $F_1$ and $F_2$ are bounded and respectively upper semi continuous and lower semi continuous functions, we can consider a point of maximum $(t^*,x^*,s^*,y^*)$ of (\ref{sup}). Without loss of generality, because we are interested in a time dependent problem, we can always assume that, for $\lambda > 0$ small enough, $F_1$ and $F_2$ are in fact sub and super solution of 
\begin{equation}\label{modifteqg}
\partial_t F + \lambda F + e^{-\lambda t} (\partial_x F)_+ \mathcal{L}[f] = 0.
\end{equation}
Because $F_1$ is a subsolution of (\ref{modifteqg}), we can use as a test function in the definition of subsolutions, the function $\phi$ defined by
\begin{equation}
\phi_1(t,x) = \frac{1}{2\epsilon}(x-y^*)^2 + \frac{1}{2\epsilon}(t-s^*)^2.
\end{equation}
By doing so, we obtain for any $\delta > 0$
\begin{equation}\label{visc1}
\begin{aligned}
\frac{1}{\epsilon}(t^* - s^*) + \lambda F_1(x^*) + e^{-\lambda t^*} \frac{1}{\epsilon}(x^* - y^*)_+(\mathcal{I}_{1,\delta}[\phi_1(t^*)](x^*) +\mathcal{I}_{2,\delta}[F_1(t^*)](x^*) ) \leq 0.
 \end{aligned}
\end{equation}
The analogue relation for $F_2$ is
\begin{equation}\label{visc2}
\begin{aligned}
\frac{1}{\epsilon}(t^* - s^*) +\lambda F_2(y^*) + e^{-\lambda s^*} \frac{1}{\epsilon}(x^* - y^*)_+( \mathcal{I}_{1,\delta}[\phi_2(s^*)](y^*)+\mathcal{I}_{2,\delta}[F_2(s^*)](y^*) ) \geq 0.
 \end{aligned}
\end{equation}
where $\phi_2$ is defined with
\begin{equation}
\phi_2(s,y) = -\frac{1}{2\epsilon}(x^*-y)^2 - \frac{1}{2\epsilon}(t^*-s)^2.
\end{equation}
Combining (\ref{visc1}) and (\ref{visc2}) yields
\begin{equation}
\begin{aligned}
&\lambda (F_1(x^*) - F_2(y^*)) + \frac{1}{\epsilon}(x^* - y^*)_+\left[ I_{1,\delta}[\phi_1](x^*) - I_{1,\delta}[\phi_2](y^*) \right] + \\
&+\frac{1}{\epsilon}(x^* - y^*)_+\left[ \mathcal{I}_{2,\delta}[F_1(t^*)](x^*) -\mathcal{I}_{2,\delta}[F_2(s^*)](y^*)  \right] \leq 0.
\end{aligned}
\end{equation}
From which we obtain
\begin{equation}
\begin{aligned}
&\lambda (F_1(x^*) - F_2(y^*)) + \frac{(x^*-y^*)_+}{\epsilon}\left[ I_{1,\delta}[\phi_1](x^*) - I_{1,\delta}[\phi_2](y^*) \right]+\\
&+\frac{(x^*-y^*)_+}{\epsilon}\left[ \int_{|z|> \delta} \frac{g(x^*,z)(F_1(t^*,x^*) - F_1(t^*,x^* + z) - F_2(s^*,y^*) + F_2(s^*,y^* + z))}{z^2} dz\right]+\\
& +\frac{(x^*-y^*)_+}{\epsilon}\int_{|z|> \delta} \frac{(g(x^*,z) - g(y^*,z))(F_2(s^*,y^*) - F_2(s^*,y^* +z))}{z^2} dz   \leq 0.
\end{aligned}
\end{equation}
By definition of $(x^*,y^*,s^*,t^*)$, since $g \geq 0$, we deduce that
\begin{equation}
\begin{aligned}
&\lambda (F_1(x^*) - F_2(y^*)) + \frac{(x^*-y^*)_+}{\epsilon}\left[ I_{1,\delta}[\phi_1](x^*) - I_{1,\delta}[\phi_2](y^*) \right]\\
& +\frac{(x^*-y^*)_+}{\epsilon}\int_{|z|> \delta} \frac{(g(x^*,z) - g(y^*,z))(F_2(s^*,y^*) - F_2(s^*,y^* +z))}{z^2} dz   \leq 0.
\end{aligned}
\end{equation}
From the uniform Lipschitz continuity of $g$ in the first variable, as well as the Lipschitz continuity of $F_2$ around $y^*$ and its global boundedness, we finally obtain that for some constant $C >0$
\begin{equation}
\begin{aligned}
&\lambda (F_1(x^*) - F_2(y^*)) + \frac{(x^*-y^*)_+}{\epsilon}\left[ I_{1,\delta}[\phi_1](x^*) - I_{1,\delta}[\phi_2](y^*) \right]\\
& +C\frac{(x^*-y^*)_+^2}{\epsilon}(1 + |\ln(\delta)|)   \leq 0.
\end{aligned}
\end{equation}
Since $\phi_1$ and $\phi_2$ are bounded in $\mathcal{C}^2$ by $\epsilon^{-1}$
\begin{equation}\label{lastineq}
\lambda (F_1(x^*) - F_2(y^*)) + 2\frac{(x^*-y^*)_+}{\epsilon}\frac{\delta}{\epsilon}  +C\frac{(x^*-y^*)_+^2}{\epsilon}(1 + |\ln(\delta)|)   \leq 0.
\end{equation}
Let us now remark that since $F_2$ is bounded, we deduce that $\epsilon^{-1}(x^*-y^*)_+$ is bounded uniformly in $\epsilon$, and thus that setting $\delta = \epsilon^2$ and letting $\epsilon \to 0$, we obtain that
\begin{equation}
\lambda \alpha \leq 0,
\end{equation}
where $\alpha, \lambda > 0$, which is thus a contradiction.

\end{proof}
\begin{Rem}\label{remdysoncomp}
Let us remark that in the Dyson case ($g = 1$), the previous proof is much simpler as no logarithmic term in $\delta$ appears and thus no Lipschitz assumption is needed.
\end{Rem}
An immediate corollary of this comparison principle is the following.
\begin{Cor}
Given an initial condition $F_0$, there exists at most one, locally in time, Lipschitz viscosity solution $F$ of (\ref{teqg+}) satisfying $F(0) = F_0$.
\end{Cor}

\subsection{Existence of viscosity solutions for smooth initial data}
In order to use the results of the previous section to prove a more general comparison principle, we need the following existence result for viscosity solutions of (\ref{teqg}) with smooth initial condition.
\begin{Prop}\label{existlip}
Assume that $F_0$ is a smooth non-decreasing bounded function. Then there exists a (unique) viscosity solution $F$ of (\ref{teqg}) satisfying $F(0) = F_0$ such that $F(t)$ is non decreasing and Lipschitz locally uniformly for $t\geq 0$.
\end{Prop}
\begin{proof}
This result can be obtained quite classically once some a priori Lipschitz estimate has been established for smooth solutions of  (\ref{teqg}). We focus in a first time on proving this estimate.\\

Let us remark that if $F$ is a smooth solution of (\ref{teqg}), then its derivative $µ$ is a solution of 
\begin{equation}\label{eqfp}
\partial_t µ + µ \mathcal{L}[µ] + \partial_x µ \mathcal{L}[F] + µ \int_{\mathbb{R}}\frac{\partial_xg(x,z)(F(t,x) - F(t,x+z))}{z^2}dz= 0 \text{ in } (0,\infty)\times \mathbb{R}.
\end{equation}
We now make some computations on the lest term of the left hand side.
\begin{equation}
\begin{aligned}
\int_{\mathbb{R}}\frac{\partial_xg(x,z)(F(t,x) - F(t,x+z))}{z^2}dz &= \int_{\mathbb{R}}\frac{\partial_xg(x,z)g(x,z)(F(t,x) - F(t,x+z))}{g(x,z)z^2}dz\\
=\int_{\mathbb{R}}\left(\frac{\partial_xg(x,z)}{g(x,z)} -\frac{\partial_xg(x,x)}{g(x,x)}\right) &\frac{g(x,z)(F(t,x) - F(t,x+z))}{z^2}dz + \frac{\partial_xg(x,x)}{g(x,x)}\mathcal{L}[F].
\end{aligned}
\end{equation}
We split the last integral into three terms, depending that $|z| \leq \delta$, $\delta < |z| \leq 1$ or $|z| > 1$ for some $\delta \in (0,\alpha_0)$. We then compute
\begin{equation}
\left|\int_{|z|\leq \delta}\left(\frac{\partial_xg(x,z)}{g(x,z)} -\frac{\partial_xg(x,x)}{g(x,x)}\right) \frac{g(x,z)(F(t,x) - F(t,x+z))}{z^2}dz \right|\leq C^2\|\partial_x F(t)\|_{\infty} \delta,
\end{equation}
\begin{equation}
\left| \int_{1 \geq |z| > \delta}\left(\frac{\partial_xg(x,z)}{g(x,z)} -\frac{\partial_xg(x,x)}{g(x,x)}\right) \frac{g(x,z)(F(t,x) - F(t,x+z))}{z^2}dz \right| \leq C\|F\|_{\infty}|\ln(\delta)|,
\end{equation}
\begin{equation}
\left| \int_{1 < |z| }\left(\frac{\partial_xg(x,z)}{g(x,z)} -\frac{\partial_xg(x,x)}{g(x,x)}\right) \frac{g(x,z)(F(t,x) - F(t,x+z))}{z^2}dz \right| \leq C\|F\|_{\infty},
\end{equation}
where $C$ is a constant given by the assumptions we made on $g$. By choosing $\delta$ as of order $\|\partial_x F(t)\|_{\infty}^{-1}$, we deduce that
\begin{equation}\label{estint}
\left|\int_{\mathbb{R}}\left(\frac{\partial_xg(x,z)}{g(x,z)} -\frac{\partial_xg(x,x)}{g(x,x)}\right) \frac{g(x,z)(F(t,x) - F(t,x+z))}{z^2}dz \right|\leq C\|F\|_{\infty}(1 + |\ln(\|\partial_x F(t)\|_{\infty})|).
\end{equation}
Let us now remark that since $F$ is a smooth solution of (\ref{teqg}), so is $\partial_tF$ and thus if $\|\partial_t F\|_{\infty}$ is finite, then the following holds :
\begin{equation}\label{estdt}
\forall t \geq 0, \|\partial_t F(t)\|_{\infty} \leq \|\partial_t F(0)\|_{\infty}.
\end{equation}
Thus we deduce from (\ref{teqg}) the following
\begin{equation}\label{estnl}
\forall t \geq 0, \|\partial_x F(t)\mathcal{L}[F(t)]\|_{\infty} \leq \|\partial_x F(0)\mathcal{L}[F_0]\|_{\infty}.
\end{equation}
Using the fact that $\frac{\partial_xg(x,0)}{g(x,0)}$ is uniformly bounded as well as (\ref{estint}) and (\ref{estnl}), we deduce from (\ref{eqfp}) that 
\begin{equation}
\frac{d}{dt}\|µ(t)\|_{\infty} \leq \|µ(t)\|_{\infty}C\|F\|_{\infty}(1 + |\ln(\|µ(t)\|_{\infty})|) + C \|\partial_x F_0\mathcal{L}[F_0]\|_{\infty}.
\end{equation}
From this last inequality, we deduce using a usual logarithmic version of Gr\"onwall's lemma that for any $t\geq 0$, there exists a constant $C_t$ depending only on $t,C$ and $\|\partial_x F_0\mathcal{L}[F_0]\|_{\infty}$ such that for any smooth solutions $F$ of (\ref{teqg})
\begin{equation}\label{lipest}
\forall s \in [0,t], \|\partial_x F(s)\|_{\infty} \leq C_t
\end{equation}

We now explain how such an estimate yields the existence of the viscosity solution in the statement of the proposition. Let us consider an approximation of the positive part $(\psi_{+,\epsilon})_{\epsilon> 0}$ which satisfies 
\begin{equation}
\begin{cases}
\psi_{+,\epsilon} \underset{\epsilon \to 0}{\longrightarrow} (\cdot)_+,\\
\forall \epsilon > 0, \psi_{+,\epsilon} \geq 0, \psi_{+,\epsilon} \in \mathcal{C}^{\infty},\\
0 \leq \psi_{+,\epsilon}' \leq 1,
\end{cases}
\end{equation}
and an analogous approximation of the negative part $(\psi_{-,\epsilon})_{\epsilon > 0}$. We also consider an approximation $(\rho_{\epsilon})_{\epsilon > 0}$ of $x \to x^{-2}$ which satisfies
\begin{equation}
\begin{cases}
\rho_{\epsilon} \underset{\epsilon \to 0}{\longrightarrow} \frac{1}{x^2} \text{ in the sense of distributions},\\
\forall \epsilon > 0, \rho_{\epsilon} \geq 0, \rho_{\epsilon} \in \mathcal{C}^{\infty}.
\end{cases}
\end{equation}
We define the operator $\mathcal{L}_{\epsilon}$ by
\begin{equation}
\mathcal{L}_{\epsilon}[\phi](x) = \int_{\mathbb{R}}g(x,z)\rho_{\epsilon}(z)(\phi(x) - \phi(x + z))dz. 
\end{equation}
Let us now consider $h,\epsilon, \delta > 0$ and the following equation
\begin{equation}\label{regteqg}
\begin{aligned}
\partial_t F + &\left(\frac{F(t,x) - F(t,x-h)}{h}\right) \psi_{+,\epsilon}(\mathcal{L}_{\delta}[F]) -\\
&- \left(\frac{F(t,x+h) - F(t,x)}{h}\right) \psi_{-,\epsilon}(\mathcal{L}_{\delta}[F]) = 0.
\end{aligned}
\end{equation}
This sort of semi discretization and regularization of (\ref{teqg}) has several key properties. First it clearly propagates the regularity of any initial condition as all the terms except $\partial_t F$ are smooth. Hence existence of solution of (\ref{regteqg}) associated to a smooth initial condition is true. Moreover, since (\ref{regteqg}) preserves the ellipticity properties of (\ref{teqg}), the estimate (\ref{lipest}) still holds for solution of (\ref{regteqg}). The proof of this fact is a direct application of the argument of the first part of this proof, namely it is useful to remark that to establish (\ref{lipest}), only the positivity of the kernel $z^{-2}$ was helpful. Hence, for $\epsilon, h, \delta >0$, and a smooth initial condition $F_0$, there exists a solution $F_{\epsilon, h, \delta}$ of (\ref{regteqg}) with initial condition $F_0$ which satisfies (\ref{lipest}) and (\ref{estdt}). Hence using Ascoli-Arzela theorem, passing to the limit $\epsilon, h, \delta \to 0$, $(F_{\epsilon, h, \delta})$ as some limit point $F$ which is a, locally in time, Lipschitz function. It is then a simple exercise that we do not detail here to verify that $F$ is indeed a non decreasing viscosity solution of (\ref{teqg}) with initial condition $F_0$.

\end{proof}
\begin{Rem}
Let us remark that in the Dyson case, the proof of the previous result is much simpler. Indeed, in this case, as the operator $\mathcal{L}$ commutes with translation, it does so with derivative and thus the proof of the formal a priori estimate is almost trivial and only requires $F_0$ to be Lipschitz continuous. Then, the formal justification of this a priori estimate and the existence of a viscosity solution can be justified by a vanishing viscosity argument, i.e. by adding a $-\epsilon \partial_{xx}$ in (\ref{teqg}) and by taking the limit $\epsilon \to 0$.
\end{Rem}

\subsection{General comparison principle and uniqueness of viscosity solutions}
We now show three results concerning general viscosity solutions of (\ref{teqg}) : propagation of monotonicity, comparison principle and uniqueness of solutions. Those three results could be proven in any order as they mostly rely on lemma \ref{complip} and proposition \ref{existlip} and not on one another.
\begin{Prop}
Let $F$ be a viscosity solution of (\ref{teqg+}) such that $F(0)$ is non-decreasing and bounded. Then $F(t)$ is increasing for all time.
\end{Prop}
\begin{proof}
Since $F(0)$ is non-decreasing and bounded, for any $\epsilon > 0$, there exists $F^+_{0,\epsilon}$ and $F^-_{0,\epsilon}$ such that those two functions are non-decreasing, bounded and Lipschitz continuous and such that
\begin{equation}
F(0) - \epsilon \leq F^-_{0,\epsilon} \leq F(0) \leq F^+_{0,\epsilon} \leq F(0) + \epsilon.
\end{equation}
Denoting by $F^{\pm}_{\epsilon}$ the viscosity solution of (\ref{teqg}) with initial condition $F^{\pm}_{0,\epsilon}$ given by proposition \ref{existlip}, we deduce from the previous inequality and from the comparison principle that for all $t\geq 0$
\begin{equation}
F(t) - \epsilon \leq F^-_{\epsilon}(t) \leq F(t) \leq F^+_{\epsilon}(t) \leq F(t) + \epsilon.
\end{equation}
Since $F^-_{\epsilon}$ and $F^+_{\epsilon}$ are non-decreasing in space for all time, we deduce by letting $\epsilon \to 0$ the required result.
\end{proof}
\begin{Prop}\label{compg}
Assume that $F_1$ and $F_2$ are respectively viscosity subsolution and supersolution of (\ref{teqg+}) such that $F_1(0) \leq F_2(0)$ and that either $F_1(0)$ or $F_2(0)$ is bounded. Then for all $t\geq 0$, $F_1(t) \leq F_2(t)$.
\end{Prop}
\begin{proof}
Since $F_2(0) \geq F_1(0)$, for any $\epsilon > 0$, there exists a non-decreasing bounded Lipschitz function $F_{0,\epsilon}$ such that
\begin{equation}
F_1(0) \leq F_{0,\epsilon} \leq F_2(0) + \epsilon.
\end{equation}
Denoting by $F_{\epsilon}$ the Lipschitz solution of (\ref{teqg}) starting from $F_{0,\epsilon}$ given by proposition \ref{existlip}, lemma \ref{complip} implies that for all $ \epsilon >0$
\begin{equation}
\forall t\geq 0, F_1(t) \leq F_{\epsilon}(t) \leq F_2(t) + \epsilon.
\end{equation}
Passing to the limit $\epsilon \to 0$ yields the required result.
\end{proof}
\begin{Theorem}\label{uniqg}
Given a non-decreasing and bounded function $F_0$, there exists at most one viscosity solution of (\ref{teqg}).
\end{Theorem}
This result is an immediate consequence of the two previous results hence we do not detail its proof. As we did in the Dyson case, we mention that the addition of a potential (or external force) does not perturb the study we just made.
\begin{Cor}\label{potentialg}
Given a probability measure $µ_0$ on $\mathbb{R}$ and a continuous real valued function $B$ which satisfies
\begin{equation}
\exists c_0 > 0, \forall x,y \in \mathbb{R}, B(x) - B(y) \geq -c_0(x-y),
\end{equation}
there exists at most one viscosity solution $F$ of 
\begin{equation}
\partial_t F + B(x)\partial_x F+ (\partial_x F) \mathcal{L}[F] = 0 \text{ in } (0,\infty)\times \mathbb{R},
\end{equation}
which satisfies $F(0,x) = µ_0((-\infty,x])$.
\end{Cor}
We do not prove the existence of a viscosity solution in this general case. Mainly we want to avoid adding another technical proof in this paper. Existence of solution should be rather standard, at least in the case in which $\mathcal{L}$ derives from an interacting particles model, by adapting the proof we provide in the Dyson case, which, as we already insisted on, does not rely strongly on the nature of the interactions between particles but mostly on the comparison principle.

\subsection{Extensions to operators without maximum principle}
In this section, we want to study the case in which the condition (\ref{hypg}) is replaced by the weaker 
\begin{equation}
\forall x \in \mathbb{R}, g(x,0) \geq 0.
\end{equation}
Mainly we show that even though the comparison principle does not hold anymore, uniqueness of viscosity solutions of (\ref{teqg}) can still be obtained. We place ourselves in the case in which $g$ satisfies (\ref{derg}), thus we shall assume that the operator $\mathcal{L}$ is formed of two part
\begin{equation}\label{newl}
\mathcal{L} = \mathcal{L}_1 + \mathcal{L}_2,
\end{equation}
where the operators $\mathcal{L}_i$ are defined by
\begin{equation}
\mathcal{L}_1[\phi](x) = \int_{\mathbb{R}}\frac{g_1(x,z)(\phi(x) - \phi(x + z))}{z^2}dz,
\end{equation}
\begin{equation}
\mathcal{L}_2[\phi](x) = \int_{\mathbb{R}}g_2(x,z)(\phi(x) - \phi(x + z))dz,
\end{equation}
where $g_1$ is a positive function satisfying (\ref{req1})-(\ref{req3}) and $g_2$ is a smooth bounded and integrable function, which is not assumed to have a sign. However, we assume that for some $C> 0$
\begin{equation}
\forall x,y \in \mathbb{R}, \left| \int_{\mathbb{R}}g_2(x,z) - g_2(y,z) dz \right|\leq C |x-y|
\end{equation}
Obviously we do not expect a comparison principle to hold in such a situation, however, the following can still be establish.
\begin{Prop}
Let $F_1$ and $F_2$ be two viscosity solution of (\ref{teqg}) (with $\mathcal{L}$ given by (\ref{newl})). Assume that
\begin{equation}
\forall t \geq 0, \sup_{0 \leq s \leq t} \|D_x F_2(s)\|_{\infty} < \infty.
\end{equation}
Then the following holds.
\begin{equation}
\|F_1(t) - F_2(t)\|_{\infty} \leq e^{C(t)t}\|F_1(0) - F_2(0)\|_{\infty},
\end{equation}
where $C(t)$ is a function depending only on $\sup_{0 \leq s \leq t} \|D_x F_2(s)\|_{\infty}$ and $g_2$.
\end{Prop}
\begin{proof}
Let us denote 
\begin{equation}
\forall t \geq 0, C(t) :=\sup_{0 \leq s \leq t} \|D_x F_2(t)\|_{\infty} < \infty.
\end{equation}
Let us define $M_+(t),M(t)$ and $M_{+,\epsilon}(t)$ by
\begin{equation}
M_+(t) := \sup \{F_1^*(s,x) - F_2(s,x)| x \in \mathbb{R},0\leq s \leq t \},
\end{equation}
\begin{equation}
M(t) := \sup \{\|F_1(s)-F_2(s)\|_{\infty}|0\leq s \leq t \},
\end{equation}
\begin{equation}\label{aidef}
M_{+,\epsilon}(t) := \sup \left\{F_1^*(s_1,x) - F_2(s_2,y) - \frac{1}{\epsilon}(x-y)^2 - \frac{1}{\epsilon}(s_1-s_2)^2| x,y \in \mathbb{R},0\leq s_1,s_2 \leq t \right\}.
\end{equation}
The same computations as in the proof of proposition \ref{complip} can be carried on except for the presence of the term
\begin{equation}
\begin{aligned}
\frac{x^* - y^*}{\epsilon}\int_{\mathbb{R}}&g_2(x^*,z)(F_1(s_1^*,x^*) - F_1(s_1^*,x^* + z) - F_2(s_2^*,y^*) + F_2(s_2^*,y^* + z))dz +\\
&+ \frac{x^* - y^*}{\epsilon}\int_{\mathbb{R}}(g_2(x^*,z) - g_2(y^*,z))(F_2(s_2^*,y^*) - F_2(s_2^*,y^* +z))dz,
\end{aligned}
\end{equation}
where $(x^*,y^*,s_1^*,s_2^*)$ is a point of maximum in (\ref{aidef}). The negative part of this term can be bounded by
\begin{equation}
2C(t) \left(C\|F_2\|_{\infty}|x^*-y^*|+ (M(t) + O(\epsilon))\int_{\mathbb{R}}(g_2)_-  \right).
\end{equation}
Taking the limit $\epsilon \to 0$ we deduce that $M_+$ is a (viscosity) solution of 
\begin{equation}
\frac{d}{dt}M_+(t) \leq 2C(t)\left(\int_{\mathbb{R}}(g_2)_-\right) M(t).
\end{equation}
By symmetry we obtain that $M$ is a solution of 
\begin{equation}
\frac{d}{dt}M(t) \leq 2C(t)\left(\int_{\mathbb{R}}(g_2)_-\right) M(t),
\end{equation}
from which the required result easily follows.
\end{proof}
As a consequence of this result, we can state.
\begin{Theorem}
Under the standing assumptions on $g_1$ and $g_2$, for any smooth ($\mathcal{C}^{1,\alpha}$) initial data $F_0$, there exists a unique Lipschitz viscosity solution of (\ref{teqg}).
\end{Theorem}
\begin{proof}
The uniqueness of such a solution immediately follows from the previous result while the existence part is a mere adaptation of proposition \ref{existlip} that we do not detail here.
\end{proof}
Finally, let us end this paper on a remark concerning other types of operators that could be of interest. As we mention in Corollary \ref{potentialg}, the addition of a potential $B$ does not raise any major difficulty. However situations in which this potential depends on the whole spectral measure could be of interest for applications and do not fall clearly in situations we already looked at in this paper. Consider for example the equation 
\begin{equation}
\partial_t F + B(x;\partial_x F)\partial_x F+ (\partial_x F) \mathcal{L}[F] = 0 \text{ in } (0,\infty)\times \mathbb{R},
\end{equation}
where $B(x; µ)$ is given by 
\begin{equation}
B(x;µ) = -\frac{x}{\epsilon}\left(\int_{\mathbb{R}}y^2µ(dy) - K\right)_+^2,
\end{equation}
with $\epsilon, K > 0$. Such a potential could model a confinement force acting on the eigenvalues once the energy of the system is too large, and thus, could prove to be of interest for various applications.

\section*{Acknowledgments}
The first, third and fourth authors have been partially supported by the Chair FDD (Institut Louis Bachelier). The fourth author has been partially supported by the Air Force Office for Scientific Research grant FA9550-18-1-0494 and the Office for Naval Research grant N000141712095.
\bibliographystyle{plainnat}
\bibliography{bibmatrix}

\end{document}